\documentclass{amsart}
\usepackage{graphicx}
\usepackage{amsmath}
\usepackage{amssymb}
\usepackage{amsthm}
\usepackage{a4wide}
\usepackage{xcolor}
\usepackage{tikz}

\theoremstyle{plain}
\newtheorem{thm}{Theorem}
\newtheorem{lem}[thm]{Lemma}

\newtheorem{cor}[thm]{Corollary}

\theoremstyle{definition}

\newtheorem{rem}{Remark}
\newtheorem{quest}{Question}

\newcommand{\R}{\ensuremath{\mathbb{R}}}

\newcommand{\T}{\ensuremath{{\mathcal T}}}

\renewcommand{\epsilon}{\varepsilon}

\begin{document}

\title[Incongruent equipartitions]{Incongruent equipartitions of the plane into quadrangles of equal perimeter}
\author{Dirk Frettl{\"o}h}
\address{Faculty of Technology, Bielefeld University, 33501 Bielefeld, Germany}
\author{Christian Richter}
\address{Institute of Mathematics, Friedrich Schiller University, 07737 Jena, Germany}
\date{\today}

\begin{abstract}
Motivated by a question of R.\ Nandakumar, we show that the Euclidean plane can be dissected into mutually incongruent convex quadrangles of the same area and the same perimeter. As a byproduct we obtain vertex-to-vertex dissections of the plane by mutually incongruent triangles of unit area that are arbitrarily close to the periodic vertex-to-vertex tiling by equilateral triangles.
\end{abstract}


\maketitle


\section{Introduction}

\subsection{Background and main results}

Sometimes problems in mathematical research are easy to formulate 
and look harmless and natural, but the answer may require a lot of effort.
The mathematical theory of tilings (also known as tesselations) provides many such 
problems. In particular, R.\ Nandakumar asks several such seemingly harmless,
but very intriguing questions about tilings in his blog \cite{nblog1}. 
Some of them have triggered a lot of research recently.
A very fruitful question of Nandakumar is ``can any convex set in the plane
be dissected into $n$ convex pieces with the same area and the same
perimeter?"  The solution to this problem requires fairly 
sophisticated tools from algebraic topology \cite{bbs, bz, kha, nr}. For a survey
see \cite{z2}. This paper is motivated by another of his problems \cite{nblog1}: 

\begin{quest}\label{quest:1}
``Can the plane be tiled by triangles of same 
area and perimeter such that no two triangles are congruent to each other?'' 
\end{quest}
Throughout this paper congruence is meant with respect to Euclidean isometries. In particular, reflections through straight lines are included.
Question \ref{quest:1} was answered in \cite{kpt1} by showing that no such tiling 
exists. Weakening the problem above  by dropping any requirement on the perimeter
makes the problem easy: 
it is not hard to find tilings of the plane by mutually incongruent triangles of unit 
area with unbounded perimeter, see \cite{nblog1}.
Hence Nandakumar asked also the following weaker version of Question \ref{quest:1}. 
\begin{quest} \label{quest:2}
``Can the plane be tiled by triangles of same area, and with
uniformly bounded perimeter, such that no two triangles are congruent to each other? If so, how small can one choose the ratio between the largest and the smallest perimeter among the triangles used?'' 
\end{quest}
The existence of such tilings by triangles with uniformly bounded perimeter was shown 
in \cite{f} (partly) and in \cite{kpt2}. In \cite{fr} even vertex-to-vertex tilings of that kind are presented. Here we improve the last result by showing that the ratio between largest and smallest perimeter can be chosen arbitrarily close to $1$.

\begin{thm} \label{thm:triangles} 
There are vertex-to-vertex tilings of the plane by pairwise incongruent non-equilateral triangles of unit area that are arbitrarily close to the periodic vertex-to-vertex tiling by equilateral triangles of unit area. 
\end{thm}

R.\ Nandakumar \cite{nblog1} proposed to consider the above questions not only for dissections into triangles, but also into convex quadrangles, pentagons or hexagons instead. During the last few years several results where obtained in this direction \cite{f,fr}, but all of those are about variations of Question \ref{quest:2}. Our main result gives a positive answer to Question~\ref{quest:1} for quadrangles. This seems to be the first positive result on partitions of the plane into incongruent convex $n$-gons of equal area and perimeter.

\begin{thm} \label{thm:quadrangles}
There are tilings of the plane by pairwise incongruent convex quadrangles of the same area and perimeter.
\end{thm}

The remainder of the present paper is organized as follows: After explaining some notations, we prove Theorem~\ref{thm:triangles} in Section~\ref{sec:triangles}. Our approach is close to Section~2 of \cite{fr}, but requires more technical effort. In Section~\ref{sec:quadrangles} Theorem~\ref{thm:quadrangles} will be inferred from Theorem~\ref{thm:triangles} by a procedure of subdividing.

\subsection{Notation}

A {\em tiling (partition, dissection, tesselation)} of a set $A \subseteq \R^2$ 
is a collection $\{T_1, T_2, \ldots \}$ 
of compact sets $T_i \subseteq \R^2$ (the {\em tiles}) that is a packing 
(i.e., the interiors of distinct tiles are disjoint) as well as a 
covering of $A$ (i.e., the union of the tiles equals $A$).
In general, shapes of tiles may be pretty complicated, but for
the purpose of this paper tiles are always convex polygons. 
A tiling is called {\em vertex-to-vertex} if the intersection of any two 
tiles is either an entire edge of both tiles, or a vertex of both tiles, or 
empty. Hence a vertex-to-vertex tiling is a polytopal cell decomposition
in the sense of \cite{z1}. 
A tiling is called \emph{periodic} if there are two linearly independent vectors in 
$\mathbb{R}^2$ such that the tiling is invariant under the translations by these vectors.
An \emph{equipartition} of the plane is a tiling of $\R^2$ such that
all tiles have the same area. One speaks of a \emph{fair partition} if all tiles 
have the same area and the same perimeter \cite{dmo}. 
We refer to \cite{gs} as a standard reference work on tilings.

Two tilings are close to each other if there is a bijection between them such that the Hausdorff distance between correspondent tiles is uniformly small. Here we use an equivalent notion that is slightly easier to handle: two tilings by triangles are called \emph{$\varepsilon$-close} to each other if there is a bijection between them such that the absolute differences between respective coordinates of respective vertices of respective triangles does not exceed $\varepsilon$.  

The symbol $\cong$ is used for congruence under Euclidean isometries including reflections. The standard inner product and the Euclidean norm are denoted by $\langle \cdot,\cdot \rangle$ and $\|\cdot\|$, respectively.


\section{Vertex-to-vertex equipartitions of the plane into almost equilateral triangles\label{sec:triangles}}

The general idea of the construction is the following: Consider the
strip $S=\R \times [-1,1]$. Tile $S$ by (pairwise congruent) triangles of 
unit area with edge lengths $\sqrt{2},\sqrt{2}$ and $2$, see Figure~\ref{fig:streifen-ungestoert}.  
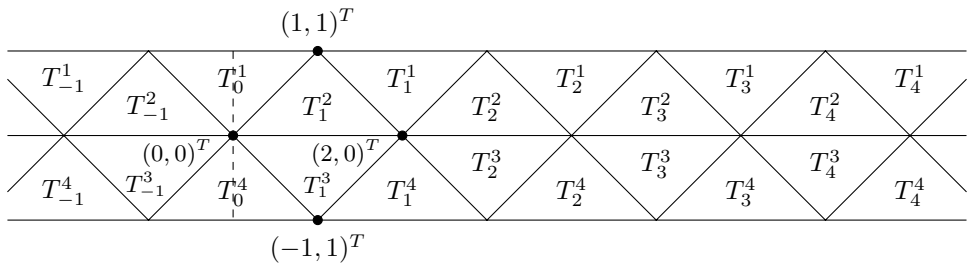
\begin{figure}[b]
\begin{center}
\begin{tikzpicture}[xscale=.375,yscale=.375]

\draw[dashed]
  (0,3)--(0,-3)
  ;

\filldraw
  (0,0) circle (1.6mm)
  (6,0) circle (1.6mm)
  (3,3) circle (1.6mm)
  (3,-3) circle (1.6mm)
  ;
 
\draw 
  (-8,3)--(26,3)
  (-8,0)--(26,0)
  (-8,-3)--(26,-3)
  (-8,2)--(-3,-3)--(3,3)--(9,-3)--(15,3)--(21,-3)--(26,2)
  (-8,-2)--(-3,3)--(3,-3)--(9,3)--(15,-3)--(21,3)--(26,-2)

  (-2,-.625) node {\small $(0,0)^T$}
  (4,-.625) node {\small $(2,0)^T$}
  (3,4) node {$(1,1)^T$}
  (3,-4) node {$(-1,1)^T$}
  
  (-6,2) node {$T_{-1}^1$}
  (-3,1) node {$T_{-1}^2$}
  (-3.1,-1.75) node {\small $T_{-1}^3$}
  (-6,-2) node {$T_{-1}^4$}
  (0,2) node {$T_0^1$}
  (0,-2) node {$T_0^4$}
  (6,2) node {$T_1^1$}
  (3,1) node {$T_1^2$}
  (3,-1.75) node {\small $T_1^3$}
  (6,-2) node {$T_1^4$}
  (12,2) node {$T_2^1$}
  (9,1) node {$T_2^2$}
  (9,-1) node {$T_2^3$}
  (12,-2) node {$T_2^4$}
  (18,2) node {$T_3^1$}
  (15,1) node {$T_3^2$}
  (15,-1) node {$T_3^3$}
  (18,-2) node {$T_3^4$}
  (24,2) node {$T_4^1$}
  (21,1) node {$T_4^2$}
  (21,-1) node {$T_4^3$}
  (24,-2) node {$T_4^4$}
  ;
\end{tikzpicture}
\end{center}

\caption{A tiling of the strip $S$ by pairwise congruent triangles.
\label{fig:streifen-ungestoert}}
\end{figure}
(At the beginning we work with isosceles right triangles in order to adopt some calculations from \cite{fr}.) Distort the tiling of $S$ by moving the vertex 
at $(0,0)^T$ to $(0,y_0)^T$ for $0<y_0<1$,
see Figure~\ref{fig:streifen-bez}.

Under the conditions that (i) the topology of the tiling is unchanged, (ii) the new tiling 
still is a tiling of $S$, (iii) the new tiling is mirror symmetric with respect to the vertical
axis $x=0$, and (iv) the tiles of the new tiling remain triangles of 
unit area, the value of $y_0$ determines all other vertices of the tiling.
See Figure \ref{fig:streifen-bez} for the situation where $y_0=\frac{1}{5}$.

In the sequel the strategy of the proof is as follows: first we obtain recursive
formulas for the coordinates of the triangles in the tiling of the strip when
$y_0$ varies, in order to control the amount of distortion of the triangles (Lemma
\ref{lem:auxiliary_estimates}). This ensures in particular that all deviations from the undistorted tilings can be kept arbitrarily small (Lemma \ref{lem:small_deviations}). 
Then we study the tiling for
$y_0 = \frac{1}{\sqrt{3}}$, having the particular property that it contains many pairwise
congruent triangles $T \cong T'$; and even stronger: it contains many triangles $T,T'$
such that $T$ or $-T$ is a translate of $T'$ (see Figure \ref{fig:1/sqrt(3)},
this property is denoted by $T \simeq T'$).
This is Lemma \ref{lem:1/sqrt(3)}. It is used in Lemmas \ref{lem:F(i,j,i',j')}
and \ref{lem:Ffinite} to show that there are only countably many $y_0$ such that
the corresponding tiling contains triangles $T,T'$ such that $T \simeq T'$. This in turn
enables us to pick a tiling $\overline{\T}^\varepsilon$ of the widened strip $\overline{S}=\mathbb{R} \times \left[-\sqrt{3},\sqrt{3}\right]$ that is $\varepsilon$-close to the tiling of $\overline{S}$ by equilateral triangles from Figure~\ref{fig:equilat_undistorted} such that for all $T,T'\in \overline{\T}^\varepsilon$ holds: $T \not\simeq T'$ (Corollary \ref{cor:T_eps}). Then again a countability argument allows us to find  
sheared copies $\big( \begin{smallmatrix} 1 & \mu_n\\ 0 & 1 \end{smallmatrix}
\big) \overline{\T}^\varepsilon$ of $\overline{\T}^\varepsilon$ such that no pair of congruent tiles occurs within them, nor in between them (Lemma \ref{lem:kong-endlich}, Corollary
\ref{cor:Tn_eps}). The tilings $\big( \begin{smallmatrix} 1 & \mu_n\\ 0 & 1 \end{smallmatrix}
\big) \overline{\T}^\varepsilon$ (of the strip $\overline{S}$) can be stacked in order to obtain
the desired vertex-to-vertex tiling of the plane that is $2\varepsilon$-close to the periodic vertex-to-vertex tiling by equilateral triangles of edge length $2$.
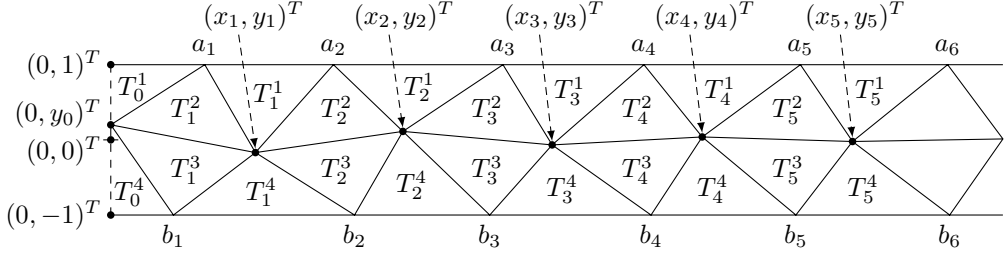
\begin{figure}
\begin{center}
\begin{tikzpicture}
\draw[dashed] 
  (0,1) -- (0,-1)
  ;
  
\fill
  (0,1) circle (.5mm)
  (0,.2) circle (.5mm)
  (0,0) circle (.5mm)
  (0,-1) circle (.5mm)
  ( 1.92307692308 , -0.169230769231 ) circle (.5mm)
  ( 3.88360118583 , 0.112523014511 ) circle (.5mm)
  ( 5.86782183927 , -0.0671455252767 ) circle (.5mm)
  ( 7.86284679278 , 0.0385390385326 ) circle (.5mm)
  ( 9.86102836451 , -0.021837269111 ) circle (.5mm)
  ;

\draw[densely dashed,-latex]
  (1.7,1.4)--( 1.92307692308 , -0.12 )
  ;
\draw[densely dashed,-latex]
  (3.7,1.4)--( 3.88360118583 , 0.16 )
  ;
\draw[densely dashed,-latex]
  (5.7,1.4)--( 5.86782183927 , -0.02 )
  ;
\draw[densely dashed,-latex]
  (7.7,1.4)--( 7.86284679278 , 0.09 )
  ;
\draw[densely dashed,-latex]
  (9.7,1.4)--( 9.86102836451 , 0.03 )
  ;

\draw
  (0,1) -- (11.86,1) 
  (0,-1) -- (11.86,-1)
  (-.1,0)--(.1,0)
  ( 0 , 0.2 ) -- ( 1.25 , 1 ) -- ( 1.92307692308 , -0.169230769231 ) -- ( 0.833333333333 , -1 ) -- ( 0 , 0.2 ) -- ( 1.92307692308 , -0.169230769231 )
  ( 1.92307692308 , -0.169230769231 ) -- ( 2.96052631579 , 1 ) -- ( 3.88360118583 , 0.112523014511 ) -- ( 3.24074074074 , -1 ) -- ( 1.92307692308 , -0.169230769231 ) -- ( 3.88360118583 , 0.112523014511 )
  ( 3.88360118583 , 0.112523014511 ) -- ( 5.21410588202 , 1 ) -- ( 5.86782183927 , -0.0671455252767 ) -- ( 5.03845636003 , -1 ) -- ( 3.88360118583 , 0.112523014511 ) -- ( 5.86782183927 , -0.0671455252767 )
  ( 5.86782183927 , -0.0671455252767 ) -- ( 7.08826451609 , 1 ) -- ( 7.86284679278 , 0.0385390385326 ) -- ( 7.18241348752 , -1 ) -- ( 5.86782183927 , -0.0671455252767 ) -- ( 7.86284679278 , 0.0385390385326 )
  ( 7.86284679278 , 0.0385390385326 ) -- ( 9.16843217775 , 1 ) -- ( 9.86102836451 , -0.021837269111 ) -- ( 9.1081956929 , -1 ) -- ( 7.86284679278 , 0.0385390385326 ) -- ( 9.86102836451 , -0.021837269111 )
  ( 9.86102836451 , -0.021837269111 ) -- ( 11.1256909903 , 1 ) -- ( 11.8605645724 , 0.0123225275777 ) -- ( 11.1528452556 , -1 ) -- ( 9.86102836451 , -0.021837269111 ) -- ( 11.8605645724 , 0.0123225275777 )
  
  (0,1) node[left] {$(0,1)^T$}
  (0,.35) node[left] {$(0,y_0)^T$}
  (0,-.15) node[left] {$(0,0)^T$}
  (0,-1) node[left] {$(0,-1)^T$}
  (.3,.7) node {$T_0^1$}
  (.25,-.7) node {$T_0^4$}
  (2.1,.6) node {$T_1^1$}
  (1,.4) node {$T_1^2$}
  (1,-.4) node {$T_1^3$}
  (2,-.7) node {$T_1^4$}
  (4.1,.7) node {$T_2^1$}
  (3,.4) node {$T_2^2$}
  (3,-.4) node {$T_2^3$}
  (4,-.6) node {$T_2^4$}
  (6.1,.65) node {$T_3^1$}
  (5,.4) node {$T_3^2$}
  (5,-.4) node {$T_3^3$}
  (6,-.65) node {$T_3^4$}
  (8.1,.65) node {$T_4^1$}
  (7,.4) node {$T_4^2$}
  (7,-.4) node {$T_4^3$}
  (8,-.65) node {$T_4^4$}
  (10.1,.65) node {$T_5^1$}
  (9,.4) node {$T_5^2$}
  (9,-.4) node {$T_5^3$}
  (10,-.65) node {$T_5^4$}
  ( 1.25 , 1 ) node[above] {$a_1$}
  ( 2.96052631579 , 1 ) node[above] {$a_2$}
  ( 5.21410588202 , 1 ) node[above] {$a_3$}
  ( 7.08826451609 , 1 ) node[above] {$a_4$}
  ( 9.16843217775 , 1 ) node[above] {$a_5$}
  ( 11.1256909903 , 1 ) node[above] {$a_6$}
  ( 0.833333333333 , -1 ) node[below] {$b_1$}
  ( 3.24074074074 , -1 ) node[below] {$b_2$}
  ( 5.03845636003 , -1 ) node[below] {$b_3$}
  ( 7.18241348752 , -1 ) node[below] {$b_4$}
  ( 9.1081956929 , -1 ) node[below] {$b_5$}
  ( 11.1528452556 , -1 ) node[below] {$b_6$}
  (1.9,1.65) node {$(x_1,y_1)^T$}
  (3.9,1.65) node {$(x_2,y_2)^T$}
  (5.9,1.65) node {$(x_3,y_3)^T$}
  (7.9,1.65) node {$(x_4,y_4)^T$}
  (9.9,1.65) node {$(x_5,y_5)^T$}
  ;

\end{tikzpicture}
\end{center}
\caption{
The distorted tiling of the half-strip $S^+$. The actual parameter for this
one is $y_0=\frac{1}{5}$. 
\label{fig:streifen-bez}}
\end{figure}

Let us start by considering the tilings of the strip $S$. 
Since we have mirror symmetry with respect to the vertical axis,
we first study the situation within the right half $S^+=[0,\infty) \times [-1,1]$ of $S$.
We need some notation, see Figure \ref{fig:streifen-bez}: 
Let $(x_i, y_i)^T$ denote the coordinates of 
the vertices along the central (distorted) line, separating the upper 
layer of triangles from the lower layer of triangles. Let $a_i$ denote
the $x$-coordinate of the vertices along the upper boundary of
the strip $S$ (the $y$-coordinate is always 1), and let $b_i$ denote
the $x$-coordinate of the vertices along the lower boundary of
the strip $S$ (the $y$-coordinate is always $-1$). Based on the parameter $y_0$, let
\begin{equation}\label{eq:start}
x_0=0, \quad y_0=y_0, \quad a_1 = \frac{1}{1-y_0}, \quad b_1=\frac{1}{1+y_0}
\end{equation}
and, for $i=1,2,\ldots$,
\begin{align}
x_{i} &= x_{i-1}+2- \frac{2(a_{i}-b_{i})y_{i-1}}{(a_{i}-x_{i-1})(1+y_{i-1})+(b_{i}-x_{i-1})(1-y_{i-1})},\label{eq:xi}\\
y_{i} &= y_{i-1}-\frac{4y_{i-1}}{(a_{i}-x_{i-1})(1+y_{i-1})+(b_{i}-x_{i-1})(1-y_{i-1})},\label{eq:yi}\\
a_{i+1} &= a_i + \frac{2}{1-y_i},\label{eq:ai}\\
b_{i+1} &= b_i + \frac{2}{1+y_i}.\label{eq:bi}
\end{align}
The choice of $a_1$ and $b_1$ ensures that the triangles $T^1_0$ and $T^4_0$ have area $1$. Formulas \eqref{eq:xi}, \eqref{eq:yi}, \eqref{eq:ai} and \eqref{eq:bi} show that $T^1_i$, $T^2_i$, $T^3_i$ and $T^4_i$ are of unit area: a simple computation yields that they imply
\begin{align}
1 &=\frac{1}{2}\det\big((x_{i},y_{i})^T-(x_{i-1},y_{i-1})^T,(a_{i},1)^T-(x_{i-1},y_{i-1})^T\big),\label{eq:T_i'}\\
1 &=\frac{1}{2}\det\big((b_{i},-1)^T-(x_{i-1},y_{i-1})^T,(x_{i},y_{i})^T-(x_{i-1},y_{i-1})^T\big),\label{eq:D_i'}\\
1 &=\frac{1}{2}(a_{i+1}-a_i)(1-y_i),\label{eq:T_i}\\
1 &=\frac{1}{2}(b_{i+1}-b_i)(1+y_i)\label{eq:D_i}
\end{align}
for $i=1,2,\ldots$ Induction shows that 
\begin{equation}\label{eq:denominator}
4i-3=\frac{1}{2}\big((x_{i-1}+a_i)(1-y_{i-1})+(x_{i-1}+b_i)(1+y_{i-1})\big)
\end{equation}
for $i=1,2,\ldots$: indeed, \eqref{eq:start} gives \eqref{eq:denominator} for $i=1$, and adding \eqref{eq:T_i'}, \eqref{eq:D_i'}, \eqref{eq:T_i} and \eqref{eq:D_i} to \eqref{eq:denominator} yields \eqref{eq:denominator} with $i$ replaced by $i+1$. By \eqref{eq:denominator}, the denominator in \eqref{eq:xi} and \eqref{eq:yi} coincides with $2(-4i+3+a_i+b_i)$. Thus
\begin{equation} \label{eq:rek-xi-yi}
x_{i} = x_{i-1}+2- \frac{(a_{i}-b_{i})y_{i-1}}{-4i+3+a_i+b_i},\quad
y_{i} = y_{i-1}-\frac{2y_{i-1}}{-4i+3+a_i+b_i}.
\end{equation}
For the sake of simplicity let $\alpha_i=a_i-(2i-1)$, $\beta_i=b_i-(2i-1)$, $\xi_i=x_i-2i$
denote the deviations of $a_i, b_i, x_i$ in the distorted tiling from 
the corresponding values in the undistorted situation. Then 
\begin{equation}\label{eq:dstart}
\xi_0=0, \quad y_0=y_0, \quad \alpha_1=\frac{y_0}{1-y_0}, \quad \beta_1=\frac{-y_0}{1+y_0}
\end{equation}
and, for $i=1,2,\ldots$, 
\begin{align}
\xi_i &= \xi_{i-1} - \frac{(\alpha_i-\beta_i) y_{i-1}}{1+\alpha_i+ \beta_i},\label{eq:dxi}\\
y_i &= y_{i-1} - \frac{2y_{i-1}}{1+\alpha_i+ \beta_i}\label{eq:dyi},\\
\alpha_{i+1} &= \alpha_i + \frac{2 y_i}{1-y_i},\label{eq:dai}\\
\beta_{i+1} &= \beta_i - \frac{2 y_i}{1+y_i}.\label{eq:dbi}
\end{align}
Formulas \eqref{eq:dstart}, \eqref{eq:dai} and \eqref{eq:dbi} show that
\[
\alpha_i+\beta_i= 2 \frac{y_0^2}{1-y_0^2} + 4 \left( \frac{y_1^2}{1-y_1^2} + 
\cdots + \frac{y_{i-1}^2}{1-y_{i-1}^2} \right).
\] 
For the sake of briefness let 
\begin{equation}\label{eq:hi}
h_i=1+\alpha_{i+1}+\beta_{i+1}=1+2 \frac{y_0^2}{1-y_0^2} + 4 \left( \frac{y_1^2}{1-y_1^2} + 
\cdots + \frac{y_i^2}{1-y_i^2} \right)  
\end{equation}
for $i=0,1,\ldots$ Then \eqref{eq:dyi} becomes
\begin{equation}\label{eq:yi+1}
y_{i+1} = \left( 1- \frac{2}{h_i} \right) y_i \quad\text{ with }\quad h_{i+1}=h_i + 4 
\frac{y_{i+1}^2}{1-y_{i+1}^2}, \quad h_0 = 1+2 \frac{y_0^2}{1-y_0^2}. 
\end{equation}

\begin{lem}\label{lem:auxiliary_estimates}
There exists $\delta > 0$ such that, for every $0 < y_0 < \delta$ and every $i \ge 0$,
\begin{eqnarray}
& h_i < 1+5\left(y_0^2+\ldots+y_i^2\right) < 2, \label{eq:hi_estimate}\\
& 1 < h_0 < h_1 < h_2 < \ldots < 2, \label{eq:hi_monotone}\\
&|y_0| > |y_1| > |y_2| > \ldots > 0 \quad\text{ and }\quad \operatorname{sign}(y_i)=(-1)^i, \label{eq:yi_monotone}\\
& \sum_{j=0}^\infty |y_j| < 4.\label{eq:yi_abs_sum}
\end{eqnarray}
\end{lem}

\begin{proof}
\emph{Step 1. We verify \eqref{eq:hi_estimate}, \eqref{eq:hi_monotone} and \eqref{eq:yi_monotone} by showing the following inductively: if $y_0>0$ satisfies 
\begin{equation}
\frac{1}{8}\left(1+2\frac{y_0^2}{1-y_0^2}\right)^2\left(1-y_0^2\right) < \frac{1}{5}
\quad\text{ and }\quad y_0 < \frac{1}{\sqrt{5}},
\label{eq:y0_assumption}
\end{equation}
then, for every $i \ge 0$,
\begin{enumerate}
\item[(a$_i$)] $h_i < 1+5\left(y_0^2+\ldots+y_i^2\right) < 2$,
\item[(b$_i$)] $1 < h_0 < h_1 < \ldots < h_i < 2$,
\item[(c$_i$)] $|y_0| > |y_1| > \ldots > |y_{i+1}| > 0$ and $\operatorname{sign}(y_{i+1})=(-1)^{i+1}$.
\end{enumerate}} 

Base case ($i=0$): By \eqref{eq:yi+1},
\[
h_0=1+2\frac{y_0^2}{1-y_0^2} \stackrel{\eqref{eq:y0_assumption}}{<} 1+2\frac{y_0^2}{1-\frac{1}{5}} < 1+5y_0^2 \stackrel{\eqref{eq:y0_assumption}}{<} 2. 
\]
This yields (a$_0$) and (b$_0$). Moreover,
\[
y_1\stackrel{\eqref{eq:yi+1}}{=} \left(1-\frac{2}{h_0}\right)y_0
\stackrel{\eqref{eq:yi+1}}{=}-\frac{1-2\frac{y_0^2}{1-y_0^2}}{1+2\frac{y_0^2}{1-y_0^2}}\,y_0=
-\frac{1-3y_0^2}{1+y_0^2}\,y_0.
\]
The second part of \eqref{eq:y0_assumption} shows that $\frac{1-3y_0^2}{1+y_0^2}\in \left(\frac{1}{3},1\right)$. This implies (c$_0$).

Step of induction: First note that
\begin{eqnarray}
y_0^2+\ldots+y_{i+1}^2 
&\stackrel{\eqref{eq:yi+1}}{=}&
\sum_{j=0}^{i+1} \left(\frac{2}{h_{j-1}}-1\right)^2\left(\frac{2}{h_{j-2}}-1\right)^2\ldots\left(\frac{2}{h_0}-1\right)^2y_0^2\nonumber\\
&\stackrel{\text{(b$_i$)}}{\le}&
\sum_{j=0}^{i+1} \left(\left(\frac{2}{h_0}-1\right)^2\right)^j y_0^2\nonumber\\
&\le&
y_0^2\,\sum_{j=0}^{\infty} \left(\left(\frac{2}{h_0}-1\right)^2\right)^j\nonumber\\
&\stackrel{\text{(b$_i$)}}{=}&
y_0^2\,\frac{1}{1-\left(\frac{2}{h_0}-1\right)^2}\nonumber\\
&\stackrel{\eqref{eq:yi+1}}{=}&
\frac{1}{8}\left(1+2\frac{y_0^2}{1-y_0^2}\right)^2\left(1-y_0^2\right)\nonumber\\
&\stackrel{\eqref{eq:y0_assumption}}{<}&
\frac{1}{5}.\label{eq:sum_yi^2}
\end{eqnarray}
By (c$_i$) and \eqref{eq:y0_assumption}, $y_0^2,\ldots,y_{i+1}^2 \le y_0^2<\frac{1}{5}$ and in turn $\frac{1}{1-y_0^2},\ldots,\frac{1}{1-y_{i+1}^2} < \frac{5}{4}$. Hence
\begin{eqnarray*}
h_{i+1}
&\stackrel{\eqref{eq:hi}}{=}&
1+2\frac{y_0^2}{1-y_0^2}+4\left(\frac{y_1^2}{1-y_1^2}+\ldots+\frac{y_{i+1}^2}{1-y_{i+1}^2}\right)\\
&<&
1+2\left(\frac{5}{4}y_0^2\right)+4\left(\frac{5}{4}y_1^2+\ldots+\frac{5}{4}y_{i+1}^2\right)\\
&<& 1+5\left(y_0^2+\ldots+y_{i+1}^2\right).
\end{eqnarray*}
This together with \eqref{eq:sum_yi^2} shows (a$_{i+1}$).

Claim (b$_{i+1}$) follows from (b$_i$) by \eqref{eq:yi+1} and (a$_{i+1}$).

Finally, note that (b$_{i+1}$) implies $-1 < 1-\frac{2}{h_{i+1}} < 0$. Therefore equation $y_{i+2}\stackrel{\eqref{eq:yi+1}}{=}\left(1-\frac{2}{h_{i+1}}\right)y_{i+1}$ together with (c$_i$) shows (c$_{i+1}$).

\emph{Step 2. Proof of \eqref{eq:yi_abs_sum}. } 
First we establish two auxiliary estimates. For every $j \ge 0$,
\begin{eqnarray}
\frac{2}{h_{j}}-1
&=&
\frac{2-h_j}{h_j}\nonumber\\
&\stackrel{\eqref{eq:yi_monotone}}{<}& 2-h_j\nonumber\\ 
&\stackrel{\eqref{eq:hi}}{=}&
1-2\frac{y_0^2}{1-y_0^2}-4\left(\frac{y_1^2}{1-y_1^2}+\ldots+\frac{y_{j}^2}{1-y_j^2}\right)\nonumber\\
&\stackrel{\eqref{eq:yi_monotone}}{<}&
1-2y_0^2-4\left(y_1^2+\ldots+y_j^2\right)\nonumber\\
&<&
1-2\left(y_0^2+\ldots+y_j^2\right).\label{eq:2/h-1}
\end{eqnarray}
Furthermore,
\begin{eqnarray*}
\left|y_{\left\lfloor\frac{1}{4y_0}\right\rfloor}\right| 
&\stackrel{\eqref{eq:yi+1},\eqref{eq:hi_monotone}}{=}& \left(\frac{2}{h_{\left\lfloor\frac{1}{4y_0}\right\rfloor-1}}-1\right)\left(\frac{2}{h_{\left\lfloor\frac{1}{4y_0}\right\rfloor-2}}-1\right)\ldots\left(\frac{2}{h_0}-1\right)y_0\\
&\stackrel{\eqref{eq:hi_monotone}}{\ge}& 
\left(\frac{2}{h_{\left\lfloor\frac{1}{4y_0}\right\rfloor-1}}-1\right)^{\left\lfloor\frac{1}{4y_0}\right\rfloor} y_0\\
&\stackrel{\eqref{eq:hi_estimate}}{>}& 
\left(\frac{2}{1+5\left(y_0^2+\ldots+y_{\left\lfloor\frac{1}{4y_0}\right\rfloor-1}^2\right)}-1\right)^{\left\lfloor\frac{1}{4y_0}\right\rfloor} y_0\\
&\stackrel{\eqref{eq:yi_monotone}}{\ge}& 
\left(\frac{2}{1+5\left\lfloor\frac{1}{4y_0}\right\rfloor 
y_0^2}-1\right)^{\left\lfloor\frac{1}{4y_0}\right\rfloor} y_0.
\end{eqnarray*}
We have $5\left\lfloor\frac{1}{4y_0}\right\rfloor y_0^2\le 5\frac{1}{4y_0}y_0^2=\frac{5}{4}y_0 \le 1$, since $0<y_0<\frac{1}{\sqrt{5}}$. Therefore,
\begin{eqnarray}
\left|y_{\left\lfloor\frac{1}{4y_0}\right\rfloor}\right| 
&\ge&
\left(\frac{2}{1+\frac{5}{4}y_0}-1\right)^{\left\lfloor\frac{1}{4y_0}\right\rfloor} y_0\nonumber\\
&=&\left(1+\frac{-\frac{5}{8}}{\frac{1}{4y_0}+\frac{5}{16}}\right)^{\left\lfloor\frac{1}{4y_0}\right\rfloor} y_0 \nonumber\\
&>& \frac{8}{15}y_0 \label{eq:y[]}
\end{eqnarray}
if $y_0 >0$ is sufficiently small, because $\lim_{y_0 \downarrow 0} \left(1+\frac{-\frac{5}{8}}{\frac{1}{4y_0}+\frac{5}{16}}\right)^{\left\lfloor\frac{1}{4y_0}\right\rfloor}=e^{-\frac{5}{8}}>\frac{8}{15}$. Now we estimate
\begin{eqnarray*}
\sum_{j=0}^\infty |y_j| 
&=& \sum_{j=0}^{\left\lfloor\frac{1}{4y_0}\right\rfloor-1} |y_j|+ \sum_{j=\left\lfloor\frac{1}{4y_0}\right\rfloor}^\infty |y_j|\\
&\stackrel{\eqref{eq:yi_monotone},\eqref{eq:yi+1},\eqref{eq:hi_monotone}}{\le}&
\left\lfloor\frac{1}{4y_0}\right\rfloor y_0+\sum_{j=\left\lfloor\frac{1}{4y_0}\right\rfloor}^\infty \left(\frac{2}{h_{j-1}}-1\right)\left(\frac{2}{h_{j-2}}-1\right)\ldots\left(\frac{2}{h_{\left\lfloor\frac{1}{4y_0}\right\rfloor}}-1\right)\left|y_{\left\lfloor\frac{1}{4y_0}\right\rfloor}\right|\\
&\stackrel{\eqref{eq:hi_monotone}}{\le}&
\frac{1}{4}+\sum_{j=\left\lfloor\frac{1}{4y_0}\right\rfloor}^\infty \left(\frac{2}{h_{\left\lfloor\frac{1}{4y_0}\right\rfloor}}-1\right)^{j-\left\lfloor\frac{1}{4y_0}\right\rfloor}\left|y_{\left\lfloor\frac{1}{4y_0}\right\rfloor}\right|\\
&=&
\frac{1}{4}+\sum_{j=0}^\infty \left(\frac{2}{h_{\left\lfloor\frac{1}{4y_0}\right\rfloor}}-1\right)^{j}\left|y_{\left\lfloor\frac{1}{4y_0}\right\rfloor}\right|\\
&\stackrel{\eqref{eq:2/h-1}}{<}&
\frac{1}{4}+\sum_{j=0}^\infty \left(1-2\left(y_0^2+\ldots+y_{\left\lfloor\frac{1}{4y_0}\right\rfloor}^2\right)\right)^{j}\left|y_{\left\lfloor\frac{1}{4y_0}\right\rfloor}\right|\\
&\stackrel{\eqref{eq:hi_estimate}}{=}&
\frac{1}{4}+\frac{1}{1-\left(1-2\left(y_0^2+\ldots+y_{\left\lfloor\frac{1}{4y_0}\right\rfloor}^2\right)\right)}\left|y_{\left\lfloor\frac{1}{4y_0}\right\rfloor}\right|\\
&=&
\frac{1}{4}+\frac{1}{2\left(y_0^2+\ldots+y_{\left\lfloor\frac{1}{4y_0}\right\rfloor}^2\right)}\left|y_{\left\lfloor\frac{1}{4y_0}\right\rfloor}\right|\\
&\stackrel{\eqref{eq:yi_monotone}}{\le}&
\frac{1}{4}+\frac{1}{2\left(\left\lfloor\frac{1}{4y_0}\right\rfloor+1\right) y_{\left\lfloor\frac{1}{4y_0}\right\rfloor}^2}\left|y_{\left\lfloor\frac{1}{4y_0}\right\rfloor}\right|\\
&=&
\frac{1}{4}+\frac{1}{2\left(\left\lfloor\frac{1}{4y_0}\right\rfloor+1\right) \left|y_{\left\lfloor\frac{1}{4y_0}\right\rfloor}\right|}\\
&\stackrel{\eqref{eq:y[]}}{<}&
\frac{1}{4}+\frac{1}{2\frac{1}{4y_0} \frac{8}{15}y_0}\\
&=&4
\end{eqnarray*}
if $y_0>0$ is sufficiently small. This is claim \eqref{eq:yi_abs_sum}.
\end{proof}

\begin{lem}\label{lem:small_deviations}
For every $\varepsilon > 0$, there exists $\delta > 0$ such that, for every $0 < y_0 < \delta$, 
all the absolute deviations $|\alpha_i|$, $|\beta_i|$, $|y_i|$ and $|\xi_i|$ of the distorted tiling $\mathcal{T}=\mathcal{T}(y_0)$ of $S$ from the undistorted tiling are uniformly bounded by $\varepsilon$.
\end{lem}

\begin{proof}
We assume $y_0$ to be sufficiently small such that claims \eqref{eq:hi_estimate}, \eqref{eq:hi_monotone}, \eqref{eq:yi_monotone} and \eqref{eq:yi_abs_sum} from Lemma~\ref{lem:auxiliary_estimates} are satisfied. We estimate
\begin{eqnarray}
\sum_{k=1}^\infty \left(|y_{2k-1}|\left|\frac{h_{2k-1}-1}{h_{2k-1}}\right|+y_{2k-1}^2\left|\frac{2}{h_{2k-1}}-1\right|\right)\hspace{-40ex}\nonumber\\
&\stackrel{\eqref{eq:hi_monotone}}{\le}&
\sum_{k=1}^\infty \left(|y_{2k-1}|(h_{2k-1}-1)+y_{2k-1}^2\right)\nonumber\\
&\stackrel{\eqref{eq:hi}}{=}&
\sum_{k=1}^\infty \left(|y_{2k-1}|\left(2\frac{y_0^2}{1-y_0^2}+4\left(\frac{y_1^2}{1-y_1^2}+\ldots+\frac{y_{2k-1}^2}{1-y_{2k-1}^2}\right)\right)+y_{2k-1}^2\right)\nonumber\\
&\le&
\sum_{k=1}^\infty \left(|y_{2k-1}|\left(4y_0^2+8\left(y_1^2+\ldots+y_{2k-1}^2\right)\right)+y_{2k-1}^2\right)\qquad (\text{since } |y_i| \stackrel{\eqref{eq:yi_monotone}}{\le} |y_0| \le \frac{1}{\sqrt{2}})\nonumber\\
&\le&
\sum_{k=1}^\infty \left(8|y_{2k-1}|\left(\sum_{j=0}^\infty y_j^2\right)+y_{2k-1}^2\right)\nonumber\\
&\le&
\left(\sum_{j=0}^\infty y_j^2\right)\left(1+8\sum_{k=1}^\infty |y_{2k-1}|\right)\nonumber\\
&\stackrel{\eqref{eq:yi_monotone}}{\le}&
y_0\left(\sum_{j=0}^\infty |y_j|\right)\left(1+8\sum_{k=0}^\infty |y_k|\right)\nonumber\\
&\stackrel{\eqref{eq:yi_abs_sum}}{<}&
132y_0.\label{eq:auxiliary_estimate_sum}
\end{eqnarray}
Now we obtain
\begin{eqnarray}
|\alpha_i| 
&\stackrel{\eqref{eq:dstart},\eqref{eq:dai}}{=}& 
\left|\frac{y_0}{1-y_0}+\sum_{j=1}^{i-1} \frac{2y_j}{1-y_j}\right|\nonumber\\
&\le& 
\left|\frac{y_0}{1-y_0}\right|+2\left|\frac{y_{i-i}}{1-y_{i-1}}\right|+2 \sum_{l=1}^\infty \left|\frac{y_{2l-1}}{1-y_{2l-1}}+\frac{y_{2l}}{1-y_{2l}}\right|\nonumber\\
&=& 
\left|\frac{y_0}{1-y_0}\right|+2\left|\frac{y_{i-i}}{1-y_{i-1}}\right|+2 \sum_{l=1}^\infty \left|\frac{y_{2l-1}+y_{2l}-2y_{2l-1}y_{2l}}{(1-y_{2l-1})(1-y_{2l})}\right|\nonumber\\
&\le& 
6y_0+8\sum_{l=1}^\infty \left|y_{2l-1}+y_{2l}-2y_{2l-1}y_{2l}\right| \qquad (\text{since } |y_i| \stackrel{\eqref{eq:yi_monotone}}{\le} |y_0|\le \frac{1}{2}) \nonumber\\
&=& 
6y_0+8\sum_{l=1}^\infty \left|y_{2l-1}+y_{2l}(1-2y_{2l-1})\right| \nonumber\\
&\stackrel{\eqref{eq:yi+1}}{=}&
6y_0+8\sum_{l=1}^\infty \left|y_{2l-1}+\left(1-\frac{2}{h_{2l-1}}\right)y_{2l-1}(1-2y_{2l-1})\right| \nonumber\\
&=&
6y_0+16\sum_{l=1}^\infty \left|y_{2l-1}\left(1-\frac{1}{h_{2l-1}}\right)+y_{2l-1}^2\left(\frac{2}{h_{2l-1}}-1\right)\right| \nonumber\\
&\le&
6y_0+16\sum_{l=1}^\infty \left(|y_{2l-1}|\left|\frac{h_{2l-1}-1}{h_{2l-1}}\right|+y_{2l-1}^2\left|\frac{2}{h_{2l-1}}-1\right|\right) \nonumber\\
&\stackrel{\eqref{eq:auxiliary_estimate_sum}}{<}&
2118y_0.\label{eq:ai_estimate}
\end{eqnarray}
Note that the second summand in the second line of the chain of (in-)equalities
above is redundant if $i$ is odd. Similarly,
\begin{eqnarray}
|\beta_i| 
&\stackrel{\eqref{eq:dstart},\eqref{eq:dbi}}{=}& 
\left|-\frac{y_0}{1+y_0}-\sum_{j=1}^{i-1} \frac{2y_j}{1+y_j}\right|\nonumber\\
&\le& 
\left|\frac{y_0}{1+y_0}\right|+2\left|\frac{y_{i-i}}{1+y_{i-1}}\right|+2 \sum_{l=1}^\infty \left|\frac{y_{2l-1}}{1+y_{2l-1}}+\frac{y_{2l}}{1+y_{2l}}\right|\nonumber\\
&=& 
\left|\frac{y_0}{1+y_0}\right|+2\left|\frac{y_{i-i}}{1+y_{i-1}}\right|+2 \sum_{l=1}^\infty \left|\frac{y_{2l-1}+y_{2l}+2y_{2l-1}y_{2l}}{(1+y_{2l-1})(1+y_{2l})}\right|\nonumber\\
&\le& 
6y_0+8\sum_{l=1}^\infty \left|y_{2l-1}+y_{2l}+2y_{2l-1}y_{2l}\right| \qquad (\text{since } |y_i| \stackrel{\eqref{eq:yi_monotone}}{\le} |y_0|\le \frac{1}{2}) \nonumber\\
&=& 
6y_0+8\sum_{l=1}^\infty \left|y_{2l-1}+y_{2l}(1+2y_{2l-1})\right| \nonumber\\
&\stackrel{\eqref{eq:yi+1}}{=}&
6y_0+8\sum_{l=1}^\infty \left|y_{2l-1}+\left(1-\frac{2}{h_{2l-1}}\right)y_{2l-1}(1+2y_{2l-1})\right| \nonumber\\
&=&
6y_0+16\sum_{l=1}^\infty \left|y_{2l-1}\left(1-\frac{1}{h_{2l-1}}\right)-y_{2l-1}^2\left(\frac{2}{h_{2l-1}}-1\right)\right| \nonumber\\
&\le&
6y_0+16\sum_{l=1}^\infty \left(|y_{2l-1}|\left|\frac{h_{2l-1}-1}{h_{2l-1}}\right|+y_{2l-1}^2\left|\frac{2}{h_{2l-1}}-1\right|\right) \nonumber\\
&\stackrel{\eqref{eq:auxiliary_estimate_sum}}{<}&
2118y_0.\label{eq:bi_estimate}
\end{eqnarray}

The claim for $|y_i|$ is already given by \eqref{eq:yi_monotone}.

For $|\xi_i|$, recall the construction of the distorted tiling (see Figure~\ref{fig:streifen-bez}). The pentagon with vertices $(0,1)^T$, $(a_i,1)^T$, $(x_i,y_i)^T$, $(b_i,-1)^T$ and $(0,-1)^T$ has area $4i-1$, because it is composed of two halve and $4i-2$ complete triangles of area $1$. Computing the area of that polygon as the sum of two areas of trapezoids with horizontal parallel edges gives
\[
4i-1=(1-y_i)\frac{a_i+x_i}{2}+(1+y_i)\frac{x_i+b_i}{2}.
\]
By $a_i=(2i-1)+\alpha_i$, $b_i=(2i-1)+\beta_i$ and $x_i=2i+\xi_i$, this yields
\[
\xi_i=\frac{1}{2}((\alpha_i-\beta_i)y_i-(\alpha_i+\beta_i)).
\]
Now, by \eqref{eq:ai_estimate}, \eqref{eq:bi_estimate} and \eqref{eq:yi_monotone}, we see that $|\xi_i|<\varepsilon$ if $y_0>0$ is sufficiently small.
\end{proof}

\begin{rem} \label{rem:sqrt3}
In \cite{fr} we have chosen initial values $y_0$ that are larger than and close to $\frac{1}{\sqrt{3}}$. Then the sequence $(y_i)_{i=0}^\infty$ is positive, monotonous and tends to $0$, but one does not obtain such strong control on the deviations $|\alpha_i|$, $|\beta_i|$, $|y_i|$ and $|\xi_i|$ as in Lemma~\ref{lem:small_deviations}. A corresponding tiling is shown in Figure~\ref{fig:streifen-monotone}. 
\begin{figure}
\begin{center}
\begin{tikzpicture}
\draw[dashed] 
  (0,1) -- (0,-1)
  ;
  
\fill
  (0,1) circle (.5mm)
  (0,.7) circle (.5mm)
  (0,0) circle (.5mm)
  (0,-1) circle (.5mm)
  ;

\draw ( 0 , 0.7 ) -- ( 3.33333333333 , 1 ) -- ( 1.34228187919 , 0.220805369128 ) -- ( 0.588235294118 , -1 ) -- ( 0 , 0.7 ) -- ( 1.34228187919 , 0.220805369128 );
\draw ( 1.34228187919 , 0.220805369128 ) -- ( 5.90008613264 , 1 ) -- ( 3.08284604993 , 0.0795615359843 ) -- ( 2.22649807587 , -1 ) -- ( 1.34228187919 , 0.220805369128 ) -- ( 3.08284604993 , 0.0795615359843 );
\draw ( 3.08284604993 , 0.0795615359843 ) -- ( 8.07296360158 , 1 ) -- ( 4.98203667813 , 0.0290793794627 ) -- ( 4.07910205752 , -1 ) -- ( 3.08284604993 , 0.0795615359843 ) -- ( 4.98203667813 , 0.0290793794627 );
\draw ( 4.98203667813 , 0.0290793794627 ) -- ( 10.1328642337 , 1 ) -- ( 6.94415799751 , 0.0106481759334 ) -- ( 6.02258672927 , -1 ) -- ( 4.98203667813 , 0.0290793794627 ) -- ( 6.94415799751 , 0.0106481759334 );
\draw 
  ( 6.94415799751 , 0.0106481759334 ) -- (11.5,.88) 
  (11.5,.8) -- ( 8.9301459945 , 0.00390007834806 ) -- ( 8.00151475548 , -1 ) -- ( 6.94415799751 , 0.0106481759334 ) -- ( 8.9301459945 , 0.00390007834806 )
  ;
\draw 
  ( 8.9301459945 , 0.00390007834806 ) -- (11.5,.495) 
  (11.5,.18) -- ( 10.9249946765 , 0.00142851870582 ) -- ( 9.99374490182 , -1 ) -- ( 8.9301459945 , 0.00390007834806 ) -- ( 10.9249946765 , 0.00142851870582 );
\draw
  (10.9249946765 , 0.00142851870582 ) -- (11.5,.114)
  ;
\draw 
  ( 10.9249946765 , 0.00142851870582 ) -- (11.5,0.001)
  ( 10.9249946765 , 0.00142851870582 ) -- (11.5,-.54)
  ;
  
\draw
  (0,1) -- (11.5,1) 
  (0,-1) -- (11.5,-1)
  (-.1,0)--(.1,0)

  (0,1) node[left] {$(0,1)^T$}
  (0,.57) node[left] {$(0,7/10)^T$}
  (0,-.0) node[left] {$(0,0)^T$}
  (0,-1) node[left] {$(0,-1)^T$}
  ;

\end{tikzpicture}
\end{center}

\caption{Suitable values $y_0>\frac{1}{\sqrt{3}}$ yield tilings of the strip $S^+$
where $y_i \downarrow 0$ monotonically as $i \to \infty$.\label{fig:streifen-monotone}}
\end{figure}
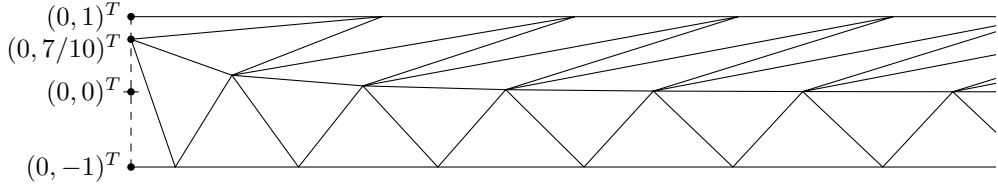
For $0 < y_0 < \frac{1}{\sqrt{3}}$, numerical evidence shows
that the behaviour of the resulting tilings is similar to Figure 
\ref{fig:streifen-bez}.
The critical case $y_0=\frac{1}{\sqrt{3}}$ is considered separately in Lemma~\ref{lem:1/sqrt(3)} and Figure~\ref{fig:1/sqrt(3)} below.
\end{rem}

Next we study congruence relations between triangles from tilings $\T=\T(y_0)$ of the strip $S$. We write $\simeq$ for congruence under the subgroup of all translations and all rotations by an angle of $\pi$. That is, two sets $A,B \in \mathbb{R}^2$ satisfy $A \simeq B$ if and only if there exist $s \in \{\pm 1\}$ and $t_1,t_2 \in \mathbb{R}$ such that $B=sA+(t_1,t_2)^T$. (Recall that we use $A \cong B$ for usual congruence under Euclidean isometries including reflections).
  
We start with an observation on the tiling $\T^\ast=\T\big(\frac{1}{\sqrt{3}}\big)$ of $S$ based on the parameter $y_0=\frac{1}{\sqrt{3}}$. The respective triangles are denoted by $T_i^{\ast j}=T_i^j\big(\frac{1}{\sqrt{3}}\big)$, see Figure~\ref{fig:1/sqrt(3)}. Here $T^{*j}_{-i}$ denotes the image of
$T^{*j}_i$ under reflection through the vertical axis.

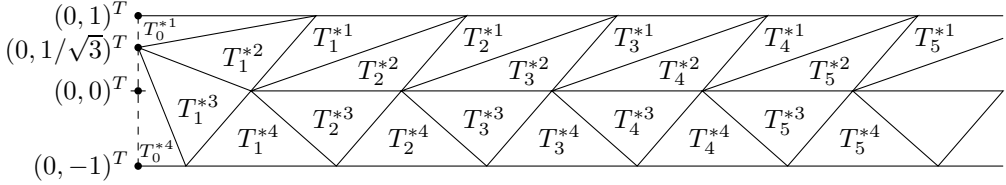
\begin{figure}
\begin{center}
\begin{tikzpicture}
\draw[dashed] 
  (0,1) -- (0,-1)
  ;
  
\fill
  (0,1) circle (.5mm)
  (0,.577) circle (.5mm)
  (0,0) circle (.5mm)
  (0,-1) circle (.5mm)

  ;

\draw ( 0 , 0.57735026919 ) -- ( 2.36602540378 , 1 ) -- ( 1.5 , 0 ) -- ( 0.633974596216 , -1 ) -- ( 0 , 0.57735026919 ) -- ( 1.5 , 0 );
\draw ( 1.5 , 0 ) -- ( 4.36602540378 , 1 ) -- ( 3.5 , 0 ) -- ( 2.63397459622 , -1 ) -- ( 1.5 , 0 ) -- ( 3.5 , 0 );
\draw ( 3.5 , 0 ) -- ( 6.36602540378 , 1 ) -- ( 5.5 , 0 ) -- ( 4.63397459622 , -1 ) -- ( 3.5 , 0 ) -- ( 5.5 , 0 );
\draw ( 5.5 , 0 ) -- ( 8.36602540378 , 1 ) -- ( 7.5 , 0 ) -- ( 6.63397459622 , -1 ) -- ( 5.5 , 0 ) -- ( 7.5 , 0 );
\draw ( 7.5 , 0 ) -- ( 10.3660254038 , 1 ) -- ( 9.5 , 0.0 ) -- ( 8.63397459622 , -1 ) -- ( 7.5 , 0 ) -- ( 9.5 , 0.0 );
\draw  ( 9.5 , 0.0 ) -- ( 11.5 , 0.0 ) -- ( 10.6339745962 , -1 ) -- ( 9.5 , 0.0 ) -- ( 11.5 , 0.7 );

\draw
  (0,1) -- (11.5,1) 
  (0,-1) -- (11.5,-1)
  (-.1,0)--(.1,0)
  (0,1) node[left] {$(0,1)^T$}
  (0,.55) node[left] {$(0,1/\sqrt{3})^T$}
  (0,0) node[left] {$(0,0)^T$}
  (0,-1) node[left] {$(0,-1)^T$}
  (.3,.82) node {\scriptsize $T_0^{\ast 1}$}
  (.25,-.8) node {\scriptsize $T_0^{\ast 4}$}
  (2.6,.7) node {$T_1^{\ast 1}$}
  (1.4,.43) node {$T_1^{\ast 2}$}
  (.8,-.25) node {$T_1^{\ast 3}$}
  (1.6,-.65) node {$T_1^{\ast 4}$}
  (4.6,.7) node {$T_2^{\ast 1}$}
  (3.2,.26) node {$T_2^{\ast 2}$}
  (2.6,-.4) node {$T_2^{\ast 3}$}
  (3.6,-.65) node {$T_2^{\ast 4}$}
  (6.6,.7) node {$T_3^{\ast 1}$}
  (5.2,.26) node {$T_3^{\ast 2}$}
  (4.6,-.4) node {$T_3^{\ast 3}$}
  (5.6,-.65) node {$T_3^{\ast 4}$}
  (8.6,.7) node {$T_4^{\ast 1}$}
  (7.2,.26) node {$T_4^{\ast 2}$}
  (6.6,-.4) node {$T_4^{\ast 3}$}
  (7.6,-.65) node {$T_4^{\ast 4}$}
  (10.6,.7) node {$T_5^{\ast 1}$}
  (9.2,.26) node {$T_5^{\ast 2}$}
  (8.6,-.4) node {$T_5^{\ast 3}$}
  (9.6,-.65) node {$T_5^{\ast 4}$}
  ;

\end{tikzpicture}
\end{center}
\caption{The tiling $\T^\ast$ of $S$ with parameter $y_0=\frac{1}{\sqrt{3}}$.
  The six congruence classes of tiles in $\T^*$ are $[T_0^{*1}],  [T_0^{*4}],  [T_1^{*2}],
  [T_1^{*3}]$ (one element each), $[T_1^{*1}]=[T_2^{*2}],  [T_1^{*4}]=[T_2^{*3}]$.
  \label{fig:1/sqrt(3)}}
\end{figure}

\begin{lem}\label{lem:1/sqrt(3)}
The coordinates of the tiling $\T^\ast$ are $x_0=0$, $y_0=\frac{1}{\sqrt{3}}$ and $x_i=2i-\frac{1}{2}$, $y_i=0$, $a_i=2i+\frac{\sqrt{3}-1}{2}$, $b_i=2i-\frac{\sqrt{3}+1}{2}$ for $i \ge 1$. Every triangle of $\T^\ast$ is congruent to one of $T_0^{\ast 1}$, $T_1^{\ast 1}$,  $T_1^{\ast 2}$, $T_1^{\ast 3}$, $T_0^{\ast 4}$ and $T_1^{\ast 4}$.
Moreover,
\begin{equation}\label{eq:simeq_ast}
T_i^{\ast j} \not \simeq T_{-i}^{\ast j} \quad\mbox{ for }\quad i=1,2,\ldots,\, j=1,2,3,4.
\end{equation}
\end{lem}

\begin{proof}
This is a direct consequence of \eqref{eq:start}-\eqref{eq:bi}.
\end{proof}

\begin{lem}\label{lem:F(i,j,i',j')}
For $y_0 \in (0,1)$ we denote the triangles from the tiling $\T=\T(y_0)$ of $S$ by $T_i^j=T_i^j(y_0)$, $(i,j) \in I=\big((\mathbb{Z} \setminus \{0\}) \times \{1,2,3,4\}\big) \cup \{(0,1),(0,4)\}$. If $(i,j),(i',j') \in I$ are such that $T_i^{\ast j} \not\simeq T_{i'}^{\ast j'}$, then the set
\[
F(i,j,i',j')=\left\{y_0 \in (0,1)\middle|\, T_i^j(y_0) \simeq T_{i'}^{j'}(y_0)\right\}
\]
is finite.
\end{lem}

\begin{proof}
We describe the triangles by their vertices, i.e., $T_i^j=\triangle\left((v^1_1,v^1_2)^T,(v^2_1,v^2_2)^T,(v^3_1,v^3_2)^T\right)$ and $T_{i'}^{j'}=\triangle\left((v'^1_1,v'^1_2)^T,(v'^2_1,v'^2_2)^T,(v'^3_1,v'^3_2)^T\right)$.
Assume that $T_i^j \simeq T_{i'}^{j'}$. Then there are $s \in \{\pm 1\}$ and $t_1,t_2 \in \mathbb{R}$ such that $T_{i'}^{j'}=sT_i^j+(t_1,t_2)^T$. The corresponding map $\varphi\big((x,y)^T\big)=s(x,y)^T+(t_1,t_2)^T$ induces a permutation $\pi$ of $\{1,2,3\}$ via $\varphi\left( (v^k_1,v^k_2)^T\right)=\big(v'^{\pi(k)}_1,v'^{\pi(k)}_2\big)^T$. Thus
\begin{equation}\label{eq:map}
s\left(v^k_1,v^k_2\right)^T+(t_1,t_2)^T=\big(v'^{\pi(k)}_1,v'^{\pi(k)}_2\big)^T \quad\mbox{ for }\quad k=1,2,3.
\end{equation}
Now we distinguish 12 situations depending on the choice of $s \in \{\pm 1\}$ and the permutation $\pi$.

\emph{Case 1: $s=1$ and $\pi$ is the identity. } From \eqref{eq:map} with $k=1$ we obtain $t_1=v'^1_1-v_1^1$ and $t_2=v'^1_2-v^1_2$. Substituting these into \eqref{eq:map} gives
\[ 
\left(v^k_1,v^k_2\right)^T+\left(v'^1_1-v_1^1,v'^1_2-v^1_2\right)^T=\left(v'^{k}_1,v'^{k}_2\right)^T \quad\mbox{ for }\quad k=1,2,3.
\]
These are six linear equations in terms of coordinates of vertices of $T_i^j$ and $T_{i'}^{j'}$. By \eqref{eq:start}-\eqref{eq:bi} these coordinates are rational functions of $y_0$. Since $T_i^{\ast j} \not\simeq T_{i'}^{\ast j'}$, at least one of these six equations fails when $y_0$ is replaced by $\frac{1}{\sqrt{3}}$. Hence that very equation is a non-trivial rational equation in $y_0$, that may have at most finitely many solutions $y_0$. Thus Case 1 applies to at most finitely many elements of $F(i,j,i',j')$.

\emph{Case 2: $s=-1$ and $\pi$ is the identity. } Now \eqref{eq:map} with $k=1$ gives
$t_1=v'^1_1+v_1^1$ and $t_2=v'^1_2+v^1_2$. We obtain
\[
-\left(v^k_1,v^k_2\right)^T+\left(v'^1_1+v_1^1,v'^1_2+v^1_2\right)^T=\left(v'^{k}_1,v'^{k}_2\right)^T \quad\mbox{ for }\quad k=1,2,3
\]
and follow the same arguments as above.
In the same way we see that each of the 12 cases yields only finitely many elements of  $F(i,j,i',j')$.
\end{proof}

\begin{lem}\label{lem:Ffinite}
Let $\frac{1}{3}\ge\delta > 0$ be as in Lemma~\ref{lem:auxiliary_estimates}. Then the set
\[
F=\left\{y_0 \in \left(0,\delta\right)\middle| \mbox{ There are distinct triangles } T,T' \in \T(y_0) \mbox{ such that } T \simeq T'.\right\}
\]
is at most countable. 
\end{lem}

\begin{proof}
Using notation from Lemma~\ref{lem:F(i,j,i',j')} we have
\[
F= \bigcup_{(i,j),(i',j') \in I, (i,j) \ne (i',j')} F(i,j,i',j') \cap \left(0,\delta\right).
\]
We shall see that all the sets $F(i,j,i',j') \cap \left(0,\delta\right)$ are finite. 

Let $y_0 \in \left(0,\delta\right)$. The \emph{vertical width} $\operatorname{vw}(T)$ of a triangle $T$ is the maximal distance between second coordinates of vertices of $T$. By Lemma~\ref{lem:auxiliary_estimates}, $(|y_i|)_{i \ge 0}$ is positive and strictly decreasing with $\operatorname{sign}(y_i)=(-1)^i$. We obtain
\[
\operatorname{vw}\left(T_i^j\right)
=\left\{ 
\begin{array}{ll}
1-y_{|i|}, & j=1,\\
1-y_{|i|-1}, & j=2,\,i \text{ even},\\
1-y_{|i|}, & j=2,\,i \text{ odd},\\
1+y_{|i|}, & j=3,\,i \text{ even},\\
1+y_{|i|-1}, & j=3,\,i \text{ odd},\\
1+y_{|i|}, & j=4.
\end{array}
\right.
\]
See Figure~\ref{fig:streifen-bez} for an illustration. 
Note that all numbers $1\pm y_{|i|}$ are different and that $T_i^j \simeq T_{i'}^{j'}$ implies $\operatorname{vw}\left(T_i^j\right)=\operatorname{vw}\left(T_{i'}^{j'}\right)$. 

Now assume that $y_0 \in F(i,j,i',j')$, i.e., $T_i^j \simeq T_{i'}^{j'}$.

\emph{Case 1: $j \in \{1,2\}$ and $j' \in \{3,4\}$ (resp.\ $j' \in \{1,2\}$ and $j \in \{3,4\}$). }
We see that $\operatorname{vw}\left(T_i^j\right) \ne \operatorname{vw}\left(T_{i'}^{j'}\right)$, which contradicts $T_i^j \simeq T_{i'}^{j'}$.
Hence $F(i,j,i',j') \cap \left(0,\delta\right)$ is empty.

\emph{Case 2: $\{j,j'\}=\{1,2\}$ or $\{j,j'\}=\{3,4\}$. } Then one of $T_i^j$ and $T_{i'}^{j'}$ has a horizontal edge and the other one has not, again a contradiction to $T_i^j \simeq T_{i'}^{j'}$. So $F(i,j,i',j') \cap \left(0,\delta\right)$ is empty, too.

\emph{Case 3: $j=j' \in \{1,4\}$ and $|i| \ne |i'|$. } We can argue as in Case 1.

\emph{Case 4: $j=j'=2$ and $|i| \ne |i'|$. } 

\emph{Subcase 4.1: $|i|$ is odd and $|i'|=|i|+1$ (or, analogously, $|i'|$ is odd and $|i|=|i'|+1$). } Now we consider vertical widths of single edges of $T_i^j$ and $T_{i'}^{j'}$ (i.e., absolute differences between the second coordinates of their endpoints). The vertical width of the edge of $T_{i'}^{j'}=T_{i'}^2$ not touching the boundary of the strip $S$ is
\begin{equation}\label{eq:vw_edge}
\left|y_{|i'|-1}-y_{|i'|}\right|=\left|y_{|i|}-y_{|i|+1}\right|=\left|y_{|i|}\right|+\left|y_{|i|+1}\right|\stackrel{\eqref{eq:yi_monotone}}{<} 2 y_0 \stackrel{(y_0<\delta)}{<} \frac{2}{3}.
\end{equation}
The vertical widths of the three edges of $T_i^j=T_i^2$ are
\[
\left.
\begin{array}{l}
\left|y_{|i|-1}-y_{|i|}\right|=\left|y_{|i|-1}\right|+\left|y_{|i|}\right|>\left|y_{|i|}\right|+\left|y_{|i|+1}\right| \\[1ex]
\left|1-y_{|i|-1}\right| \ge 1-\left|y_{|i|-1}\right| \stackrel{\eqref{eq:yi_monotone}}{\ge} 1-\left|y_{0}\right|\stackrel{(y_0<\delta)}{>} \frac{2}{3} \\[1ex] 
\left|1-y_{|i|}\right| \ge 1-\left|y_{|i|}\right|
\stackrel{\eqref{eq:yi_monotone}}{\ge} 1-\left|y_{0}\right|\stackrel{(y_0<\delta)}{>} \frac{2}{3}
\end{array}
\right\}
\stackrel{\eqref{eq:vw_edge}}{\ge}\left|y_{|i'|-1}-y_{|i'|}\right|.
\]
We see that the vertical width of one edge of $T_{i'}^{j'}$ is different from that of all edges of $T_i^j$, a contradiction to $T_i^j \simeq T_{i'}^{j'}$. So $F(i,j,i',j') \cap \left(0,\delta\right)$ is empty.

\emph{Subcase 4.2: The assumption of Subcase 4.1 fails. } We can argue as in Case 1.

\emph{Case 5: $j=j'=3$ and $|i| \ne |i'|$. } We can argue as in Case 4.

\emph{Case 6: $j=j'$ and $|i|=|i'|$. } Since  $(i,j) \ne (i',j')$, we have $i'=-i \ne 0$. By \eqref{eq:simeq_ast} and Lemma~\ref{lem:F(i,j,i',j')}, $F(i,j,i',j')$ is finite. This completes the proof.
\end{proof}

Now we switch to tilings by almost equilateral triangles. For that, we apply the linear map $\left(\begin{smallmatrix} 1 & 0 \\ 0 & \sqrt{3} \end{smallmatrix}\right)$ to the strip $S$ and to the triangles constructed so far. The undistorted tiling $\overline{\T}^0$ of $\overline{S}=\left(\begin{smallmatrix} 1 & 0 \\ 0 & \sqrt{3} \end{smallmatrix}\right)S=\mathbb{R} \times \left[-\sqrt{3},\sqrt{3}\right]$ by equilateral triangles of edge length $2$ and area $\sqrt{3}$ is illustrated in Figure~\ref{fig:equilat_undistorted}.
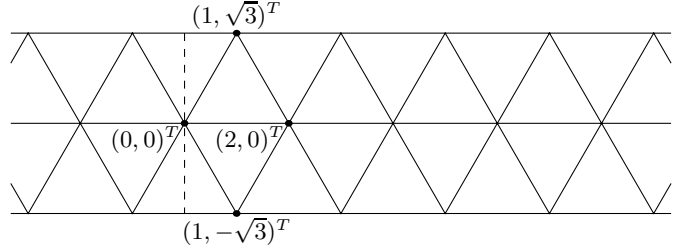
\begin{figure}
\begin{center}
\begin{tikzpicture}[xscale=.231,yscale=.2]

\draw[dashed]
  (0,6)--(0,-6)
  ;
  
\fill
  (0,0) circle (2.2mm)
  (3,6) circle (2.2mm)
  (3,-6) circle (2.2mm)
  (6,0) circle (2.2mm)
  ;
  
\draw 
  (-10,6)--(28,6)
  (-10,0)--(28,0)
  (-10,-6)--(28,-6)
  (-10,4)--(-9,6)--(-3,-6)--(3,6)--(9,-6)--(15,6)--(21,-6)--(27,6)--(28,4)
  (-10,-4)--(-9,-6)--(-3,6)--(3,-6)--(9,6)--(15,-6)--(21,6)--(27,-6)--(28,-4)

  (-2.25,-1.1) node {\small $(0,0)^T$}
  (3.75,-1.1) node {\small $(2,0)^T$}
  (3,7.2) node {\small $(1,\sqrt{3})^T$}
  (3,-7.2) node {\small $(1,-\sqrt{3})^T$}
  ;
\end{tikzpicture}
\end{center}
\caption{The tiling $\overline{\T}^0$ of the strip $\overline{S}$ by equilateral triangles of area $\sqrt{3}$.\label{fig:equilat_undistorted}}
\end{figure}

\begin{cor}\label{cor:T_eps}
For every $\varepsilon > 0$, there is a vertex-to-vertex tiling $\overline{\T}^\varepsilon$ of the strip $\overline{S}=\mathbb{R} \times \left[-\sqrt{3},\sqrt{3}\right]$ by triangles of area $\sqrt{3}$ that is $\varepsilon$-close to the undistorted tiling $\overline{\T}^0$ of $\overline{S}$ displayed in Figure~\ref{fig:equilat_undistorted} such that $T \not \simeq T'$ for all $T,T' \in \overline{\T}^\varepsilon$ with $T \ne T'$.
\end{cor}

\begin{proof}
The linear map $\left(\begin{smallmatrix} 1 & 0 \\ 0 & \sqrt{3} \end{smallmatrix}\right)$ transforms $S$ into $\overline{S}$ and the tiling from Figure~\ref{fig:streifen-ungestoert} into that from Figure~\ref{fig:equilat_undistorted}. Since $\left(\begin{smallmatrix} 1 & 0 \\ 0 & \sqrt{3} \end{smallmatrix}\right)T \simeq \left(\begin{smallmatrix} 1 & 0 \\ 0 & \sqrt{3} \end{smallmatrix}\right)T'$ if and only if $T \simeq T'$, if suffices to show the following:
for every $\varepsilon > 0$, there exists $y_0 > 0$ such that the tiling $\T(y_0)$ of $S$ satisfies $T \not \simeq T'$ for all $T,T' \in \T(y_0)$ with $T \ne T'$ and that all the corresponding absolute deviations $|\alpha_i|$, $|\beta_i|$, $|y_i|$ and $|\xi_i|$ are uniformly bounded by $\varepsilon$.

We fix $\delta \in \left(0,\frac{1}{3}\right]$ such that the claims of Lemmas~\ref{lem:small_deviations} and \ref{lem:Ffinite} are satisfied. By Lemma~\ref{lem:Ffinite}, we can fix $y_0 \in (0,\delta) \setminus F$. Then Lemma~\ref{lem:Ffinite} says that $T \not \simeq T'$ for all $T,T' \in \T(y_0)$ with $T \ne T'$.
Lemma~\ref{lem:small_deviations} shows that $|\alpha_i|$, $|\beta_i|$, $|y_i|$ and $|\xi_i|$ are bounded by $\varepsilon$.
\end{proof}

The final tiling of $\mathbb{R}^2$ will be obtained by stacking sheared copies $\left( \begin{smallmatrix}1 & \mu \\ 0 & 1 
\end{smallmatrix} \right) \overline{\T}^\varepsilon$ of $\overline{\T}^\varepsilon$, see Figure \ref{fig:final}.
\begin{figure}
\begin{center}
\begin{tikzpicture}[xscale=.4,yscale=.693]

\draw
  (11,4) node[right] {\Large $\Bigg\}\varphi_4\left(\overline{\T}_4^\varepsilon\right)$}
  (11,2) node[right] {\Large $\Bigg\}\varphi_2\left(\overline{\T}_2^\varepsilon\right)$}
  (11,0) node[right] {\Large $\Bigg\}\overline{\T}_1^\varepsilon$}
  (11,-2) node[right] {\Large $\Bigg\}\varphi_{3}\left(\overline{\T}_{3}^\varepsilon\right)$}
  (11,-4) node[right] {\Large $\Bigg\}\varphi_{5}\left(\overline{\T}_{5}^\varepsilon\right)$}
  ;

\begin{scope}
\clip (-11,-5) rectangle (11,5);

\begin{scope}[cm={1,0,.3,1,(0,0)}]
  \draw[densely dashed] 
  (0,1) -- (0,-1)
  ;
  \draw
  (-11.86,1) -- (11.86,1) 
  (-11.86,-1) -- (11.86,-1)

  ( 0 , 0.2 ) -- ( 1.25 , 1 ) -- ( 1.92307692308 , -0.169230769231 ) -- ( 2.96052631579 , 1 ) -- ( 3.88360118583 , 0.112523014511 ) -- ( 5.21410588202 , 1 ) -- ( 5.86782183927 , -0.0671455252767 ) -- ( 7.08826451609 , 1 ) -- ( 7.86284679278 , 0.0385390385326 ) -- ( 9.16843217775 , 1 ) -- ( 9.86102836451 , -0.021837269111 ) -- ( 11.1256909903 , 1 ) -- ( 11.8605645724 , 0.0123225275777 )

  ( 0 , 0.2 ) -- ( -1.25 , 1 ) -- ( -1.92307692308 , -0.169230769231 ) -- ( -2.96052631579 , 1 ) -- ( -3.88360118583 , 0.112523014511 ) -- ( -5.21410588202 , 1 ) -- ( -5.86782183927 , -0.0671455252767 ) -- ( -7.08826451609 , 1 ) -- ( -7.86284679278 , 0.0385390385326 ) -- ( -9.16843217775 , 1 ) -- ( -9.86102836451 , -0.021837269111 ) -- ( -11.1256909903 , 1 ) -- ( -11.8605645724 , 0.0123225275777 )

  ( 0 , 0.2 ) -- ( 0.833333333333 , -1 ) -- ( 1.92307692308 , -0.169230769231 ) -- ( 3.24074074074 , -1 ) -- ( 3.88360118583 , 0.112523014511 ) -- ( 5.03845636003 , -1 ) -- ( 5.86782183927 , -0.0671455252767 ) -- ( 7.18241348752 , -1 ) -- ( 7.86284679278 , 0.0385390385326 ) -- ( 9.1081956929 , -1 ) -- ( 9.86102836451 , -0.021837269111 ) -- ( 11.1528452556 , -1 ) -- ( 11.8605645724 , 0.0123225275777 )

  ( 0 , 0.2 ) -- ( -0.833333333333 , -1 ) -- ( -1.92307692308 , -0.169230769231 ) -- ( -3.24074074074 , -1 ) -- ( -3.88360118583 , 0.112523014511 ) -- ( -5.03845636003 , -1 ) -- ( -5.86782183927 , -0.0671455252767 ) -- ( -7.18241348752 , -1 ) -- ( -7.86284679278 , 0.0385390385326 ) -- ( -9.1081956929 , -1 ) -- ( -9.86102836451 , -0.021837269111 ) -- ( -11.1528452556 , -1 ) -- ( -11.8605645724 , 0.0123225275777 )
  
  ( 0 , 0.2 ) -- ( 1.92307692308 , -0.169230769231 ) -- ( 3.88360118583 , 0.112523014511 ) -- ( 5.86782183927 , -0.0671455252767 ) -- ( 7.86284679278 , 0.0385390385326 ) -- ( 9.86102836451 , -0.021837269111 ) -- ( 11.8605645724 , 0.0123225275777 )

  ( 0 , 0.2 ) -- ( -1.92307692308 , -0.169230769231 ) -- ( -3.88360118583 , 0.112523014511 ) -- ( -5.86782183927 , -0.0671455252767 ) -- ( -7.86284679278 , 0.0385390385326 ) -- ( -9.86102836451 , -0.021837269111 ) -- ( -11.8605645724 , 0.0123225275777 )
  ;
\end{scope}
  
\begin{scope}[cm={1,0,-.15,-1,(.45,2)}]
  \draw[densely dashed] 
  (0,1) -- (0,-1)
  ;
  \draw
  (-11.86,1) -- (11.86,1) 
  (-11.86,-1) -- (11.86,-1)

  ( 0 , 0.2 ) -- ( 1.25 , 1 ) -- ( 1.92307692308 , -0.169230769231 ) -- ( 2.96052631579 , 1 ) -- ( 3.88360118583 , 0.112523014511 ) -- ( 5.21410588202 , 1 ) -- ( 5.86782183927 , -0.0671455252767 ) -- ( 7.08826451609 , 1 ) -- ( 7.86284679278 , 0.0385390385326 ) -- ( 9.16843217775 , 1 ) -- ( 9.86102836451 , -0.021837269111 ) -- ( 11.1256909903 , 1 ) -- ( 11.8605645724 , 0.0123225275777 )

  ( 0 , 0.2 ) -- ( -1.25 , 1 ) -- ( -1.92307692308 , -0.169230769231 ) -- ( -2.96052631579 , 1 ) -- ( -3.88360118583 , 0.112523014511 ) -- ( -5.21410588202 , 1 ) -- ( -5.86782183927 , -0.0671455252767 ) -- ( -7.08826451609 , 1 ) -- ( -7.86284679278 , 0.0385390385326 ) -- ( -9.16843217775 , 1 ) -- ( -9.86102836451 , -0.021837269111 ) -- ( -11.1256909903 , 1 ) -- ( -11.8605645724 , 0.0123225275777 )

  ( 0 , 0.2 ) -- ( 0.833333333333 , -1 ) -- ( 1.92307692308 , -0.169230769231 ) -- ( 3.24074074074 , -1 ) -- ( 3.88360118583 , 0.112523014511 ) -- ( 5.03845636003 , -1 ) -- ( 5.86782183927 , -0.0671455252767 ) -- ( 7.18241348752 , -1 ) -- ( 7.86284679278 , 0.0385390385326 ) -- ( 9.1081956929 , -1 ) -- ( 9.86102836451 , -0.021837269111 ) -- ( 11.1528452556 , -1 ) -- ( 11.8605645724 , 0.0123225275777 )

  ( 0 , 0.2 ) -- ( -0.833333333333 , -1 ) -- ( -1.92307692308 , -0.169230769231 ) -- ( -3.24074074074 , -1 ) -- ( -3.88360118583 , 0.112523014511 ) -- ( -5.03845636003 , -1 ) -- ( -5.86782183927 , -0.0671455252767 ) -- ( -7.18241348752 , -1 ) -- ( -7.86284679278 , 0.0385390385326 ) -- ( -9.1081956929 , -1 ) -- ( -9.86102836451 , -0.021837269111 ) -- ( -11.1528452556 , -1 ) -- ( -11.8605645724 , 0.0123225275777 )
  
  ( 0 , 0.2 ) -- ( 1.92307692308 , -0.169230769231 ) -- ( 3.88360118583 , 0.112523014511 ) -- ( 5.86782183927 , -0.0671455252767 ) -- ( 7.86284679278 , 0.0385390385326 ) -- ( 9.86102836451 , -0.021837269111 ) -- ( 11.8605645724 , 0.0123225275777 )

  ( 0 , 0.2 ) -- ( -1.92307692308 , -0.169230769231 ) -- ( -3.88360118583 , 0.112523014511 ) -- ( -5.86782183927 , -0.0671455252767 ) -- ( -7.86284679278 , 0.0385390385326 ) -- ( -9.86102836451 , -0.021837269111 ) -- ( -11.8605645724 , 0.0123225275777 )
  ;
\end{scope}

\begin{scope}[cm={1,0,-.15,-1,(-.45,-2)}]
  \draw[densely dashed] 
  (0,1) -- (0,-1)
  ;
  \draw
  (-11.86,1) -- (11.86,1) 
  (-11.86,-1) -- (11.86,-1)

  ( 0 , 0.2 ) -- ( 1.25 , 1 ) -- ( 1.92307692308 , -0.169230769231 ) -- ( 2.96052631579 , 1 ) -- ( 3.88360118583 , 0.112523014511 ) -- ( 5.21410588202 , 1 ) -- ( 5.86782183927 , -0.0671455252767 ) -- ( 7.08826451609 , 1 ) -- ( 7.86284679278 , 0.0385390385326 ) -- ( 9.16843217775 , 1 ) -- ( 9.86102836451 , -0.021837269111 ) -- ( 11.1256909903 , 1 ) -- ( 11.8605645724 , 0.0123225275777 )

  ( 0 , 0.2 ) -- ( -1.25 , 1 ) -- ( -1.92307692308 , -0.169230769231 ) -- ( -2.96052631579 , 1 ) -- ( -3.88360118583 , 0.112523014511 ) -- ( -5.21410588202 , 1 ) -- ( -5.86782183927 , -0.0671455252767 ) -- ( -7.08826451609 , 1 ) -- ( -7.86284679278 , 0.0385390385326 ) -- ( -9.16843217775 , 1 ) -- ( -9.86102836451 , -0.021837269111 ) -- ( -11.1256909903 , 1 ) -- ( -11.8605645724 , 0.0123225275777 )

  ( 0 , 0.2 ) -- ( 0.833333333333 , -1 ) -- ( 1.92307692308 , -0.169230769231 ) -- ( 3.24074074074 , -1 ) -- ( 3.88360118583 , 0.112523014511 ) -- ( 5.03845636003 , -1 ) -- ( 5.86782183927 , -0.0671455252767 ) -- ( 7.18241348752 , -1 ) -- ( 7.86284679278 , 0.0385390385326 ) -- ( 9.1081956929 , -1 ) -- ( 9.86102836451 , -0.021837269111 ) -- ( 11.1528452556 , -1 ) -- ( 11.8605645724 , 0.0123225275777 )

  ( 0 , 0.2 ) -- ( -0.833333333333 , -1 ) -- ( -1.92307692308 , -0.169230769231 ) -- ( -3.24074074074 , -1 ) -- ( -3.88360118583 , 0.112523014511 ) -- ( -5.03845636003 , -1 ) -- ( -5.86782183927 , -0.0671455252767 ) -- ( -7.18241348752 , -1 ) -- ( -7.86284679278 , 0.0385390385326 ) -- ( -9.1081956929 , -1 ) -- ( -9.86102836451 , -0.021837269111 ) -- ( -11.1528452556 , -1 ) -- ( -11.8605645724 , 0.0123225275777 )
  
  ( 0 , 0.2 ) -- ( 1.92307692308 , -0.169230769231 ) -- ( 3.88360118583 , 0.112523014511 ) -- ( 5.86782183927 , -0.0671455252767 ) -- ( 7.86284679278 , 0.0385390385326 ) -- ( 9.86102836451 , -0.021837269111 ) -- ( 11.8605645724 , 0.0123225275777 )

  ( 0 , 0.2 ) -- ( -1.92307692308 , -0.169230769231 ) -- ( -3.88360118583 , 0.112523014511 ) -- ( -5.86782183927 , -0.0671455252767 ) -- ( -7.86284679278 , 0.0385390385326 ) -- ( -9.86102836451 , -0.021837269111 ) -- ( -11.8605645724 , 0.0123225275777 )
  ;
\end{scope}

\begin{scope}[cm={1,0,.075,1,(.675,4)}]
  \draw[densely dashed] 
  (0,1) -- (0,-1)
  ;
  \draw
  (-11.86,1) -- (11.86,1) 
  (-11.86,-1) -- (11.86,-1)

  ( 0 , 0.2 ) -- ( 1.25 , 1 ) -- ( 1.92307692308 , -0.169230769231 ) -- ( 2.96052631579 , 1 ) -- ( 3.88360118583 , 0.112523014511 ) -- ( 5.21410588202 , 1 ) -- ( 5.86782183927 , -0.0671455252767 ) -- ( 7.08826451609 , 1 ) -- ( 7.86284679278 , 0.0385390385326 ) -- ( 9.16843217775 , 1 ) -- ( 9.86102836451 , -0.021837269111 ) -- ( 11.1256909903 , 1 ) -- ( 11.8605645724 , 0.0123225275777 )

  ( 0 , 0.2 ) -- ( -1.25 , 1 ) -- ( -1.92307692308 , -0.169230769231 ) -- ( -2.96052631579 , 1 ) -- ( -3.88360118583 , 0.112523014511 ) -- ( -5.21410588202 , 1 ) -- ( -5.86782183927 , -0.0671455252767 ) -- ( -7.08826451609 , 1 ) -- ( -7.86284679278 , 0.0385390385326 ) -- ( -9.16843217775 , 1 ) -- ( -9.86102836451 , -0.021837269111 ) -- ( -11.1256909903 , 1 ) -- ( -11.8605645724 , 0.0123225275777 )

  ( 0 , 0.2 ) -- ( 0.833333333333 , -1 ) -- ( 1.92307692308 , -0.169230769231 ) -- ( 3.24074074074 , -1 ) -- ( 3.88360118583 , 0.112523014511 ) -- ( 5.03845636003 , -1 ) -- ( 5.86782183927 , -0.0671455252767 ) -- ( 7.18241348752 , -1 ) -- ( 7.86284679278 , 0.0385390385326 ) -- ( 9.1081956929 , -1 ) -- ( 9.86102836451 , -0.021837269111 ) -- ( 11.1528452556 , -1 ) -- ( 11.8605645724 , 0.0123225275777 )

  ( 0 , 0.2 ) -- ( -0.833333333333 , -1 ) -- ( -1.92307692308 , -0.169230769231 ) -- ( -3.24074074074 , -1 ) -- ( -3.88360118583 , 0.112523014511 ) -- ( -5.03845636003 , -1 ) -- ( -5.86782183927 , -0.0671455252767 ) -- ( -7.18241348752 , -1 ) -- ( -7.86284679278 , 0.0385390385326 ) -- ( -9.1081956929 , -1 ) -- ( -9.86102836451 , -0.021837269111 ) -- ( -11.1528452556 , -1 ) -- ( -11.8605645724 , 0.0123225275777 )
  
  ( 0 , 0.2 ) -- ( 1.92307692308 , -0.169230769231 ) -- ( 3.88360118583 , 0.112523014511 ) -- ( 5.86782183927 , -0.0671455252767 ) -- ( 7.86284679278 , 0.0385390385326 ) -- ( 9.86102836451 , -0.021837269111 ) -- ( 11.8605645724 , 0.0123225275777 )

  ( 0 , 0.2 ) -- ( -1.92307692308 , -0.169230769231 ) -- ( -3.88360118583 , 0.112523014511 ) -- ( -5.86782183927 , -0.0671455252767 ) -- ( -7.86284679278 , 0.0385390385326 ) -- ( -9.86102836451 , -0.021837269111 ) -- ( -11.8605645724 , 0.0123225275777 )
  ;
\end{scope}

\begin{scope}[cm={1,0,.075,1,(-.675,-4)}]
  \draw[densely dashed] 
  (0,1) -- (0,-1)
  ;
  \draw
  (-11.86,1) -- (11.86,1) 
  (-11.86,-1) -- (11.86,-1)

  ( 0 , 0.2 ) -- ( 1.25 , 1 ) -- ( 1.92307692308 , -0.169230769231 ) -- ( 2.96052631579 , 1 ) -- ( 3.88360118583 , 0.112523014511 ) -- ( 5.21410588202 , 1 ) -- ( 5.86782183927 , -0.0671455252767 ) -- ( 7.08826451609 , 1 ) -- ( 7.86284679278 , 0.0385390385326 ) -- ( 9.16843217775 , 1 ) -- ( 9.86102836451 , -0.021837269111 ) -- ( 11.1256909903 , 1 ) -- ( 11.8605645724 , 0.0123225275777 )

  ( 0 , 0.2 ) -- ( -1.25 , 1 ) -- ( -1.92307692308 , -0.169230769231 ) -- ( -2.96052631579 , 1 ) -- ( -3.88360118583 , 0.112523014511 ) -- ( -5.21410588202 , 1 ) -- ( -5.86782183927 , -0.0671455252767 ) -- ( -7.08826451609 , 1 ) -- ( -7.86284679278 , 0.0385390385326 ) -- ( -9.16843217775 , 1 ) -- ( -9.86102836451 , -0.021837269111 ) -- ( -11.1256909903 , 1 ) -- ( -11.8605645724 , 0.0123225275777 )

  ( 0 , 0.2 ) -- ( 0.833333333333 , -1 ) -- ( 1.92307692308 , -0.169230769231 ) -- ( 3.24074074074 , -1 ) -- ( 3.88360118583 , 0.112523014511 ) -- ( 5.03845636003 , -1 ) -- ( 5.86782183927 , -0.0671455252767 ) -- ( 7.18241348752 , -1 ) -- ( 7.86284679278 , 0.0385390385326 ) -- ( 9.1081956929 , -1 ) -- ( 9.86102836451 , -0.021837269111 ) -- ( 11.1528452556 , -1 ) -- ( 11.8605645724 , 0.0123225275777 )

  ( 0 , 0.2 ) -- ( -0.833333333333 , -1 ) -- ( -1.92307692308 , -0.169230769231 ) -- ( -3.24074074074 , -1 ) -- ( -3.88360118583 , 0.112523014511 ) -- ( -5.03845636003 , -1 ) -- ( -5.86782183927 , -0.0671455252767 ) -- ( -7.18241348752 , -1 ) -- ( -7.86284679278 , 0.0385390385326 ) -- ( -9.1081956929 , -1 ) -- ( -9.86102836451 , -0.021837269111 ) -- ( -11.1528452556 , -1 ) -- ( -11.8605645724 , 0.0123225275777 )
  
  ( 0 , 0.2 ) -- ( 1.92307692308 , -0.169230769231 ) -- ( 3.88360118583 , 0.112523014511 ) -- ( 5.86782183927 , -0.0671455252767 ) -- ( 7.86284679278 , 0.0385390385326 ) -- ( 9.86102836451 , -0.021837269111 ) -- ( 11.8605645724 , 0.0123225275777 )

  ( 0 , 0.2 ) -- ( -1.92307692308 , -0.169230769231 ) -- ( -3.88360118583 , 0.112523014511 ) -- ( -5.86782183927 , -0.0671455252767 ) -- ( -7.86284679278 , 0.0385390385326 ) -- ( -9.86102836451 , -0.021837269111 ) -- ( -11.8605645724 , 0.0123225275777 )
  ;
  \end{scope}
\end{scope}  
\end{tikzpicture}
\end{center}

\caption{Stacked images of $\overline{\T}_n^\varepsilon$, $n \ge 1$, tile the plane. The $\overline{\T}_n^\varepsilon$
are sheared copies of $\overline{\T}^\varepsilon$. The $\varphi_n(\overline{\T}_n^\varepsilon)$ are images of $\overline{\T}_n^\varepsilon$ under appropriate translations (and possibly reflections).
\label{fig:final}} 
\end{figure}
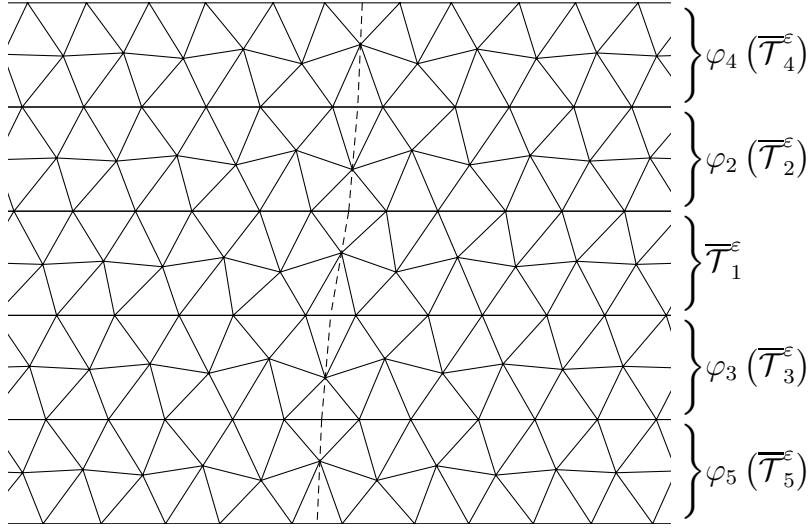
In order to make sure that almost every shear mapping of $\overline{\T}^\varepsilon$ produces (i) mutually incongruent triangles that are (ii) different from countably many 
prescribed shapes and (iii) not equilateral, we need the following result.

\begin{lem} \label{lem:kong-endlich}
Let $T$ and $T'$ be triangles.
\begin{enumerate}
\item[(a)] If $T \not\simeq T'$ then the set
$\left\{ \mu \in \R \mid \left( \begin{smallmatrix}1 & \mu \\ 0 & 1 
\end{smallmatrix} \right) T \cong \left( \begin{smallmatrix}1 & \mu \\ 0 & 1 
\end{smallmatrix} \right) T' \right\}$ is finite.
\item[(b)] The set $\left\{ \mu \in \R \mid
\left( \begin{smallmatrix}1 & \mu \\ 0 & 1 \end{smallmatrix} \right) T
\cong T' \right\}$ is finite.
\item[(c)] The set $\left\{ \mu \in \R \mid
\left( \begin{smallmatrix}1 & \mu \\ 0 & 1 \end{smallmatrix} \right) T
\text{ is equilateral.} \right\}$ is finite.
\end{enumerate}
\end{lem}

\begin{proof}
(a) Since $T \not\simeq T'$, there exists an 
edge $e_0$ of $T$ that is a translate of neither of the three 
edges $e_1,e_2,e_3$ of $T'$. Then
\[
\left\{ \mu \in \R \mid \left( \begin{smallmatrix}1 & \mu \\ 0 & 1 
\end{smallmatrix} \right) T \cong \left( \begin{smallmatrix}1 & \mu \\ 0 & 1 
\end{smallmatrix} \right) T' \right\} 
\subseteq \big\{ \mu \in \R \big| \left\|\left( \begin{smallmatrix}1 & \mu \\ 0 & 1 
\end{smallmatrix} \right) e_0 \right\| 
\in \left\{ \left\|\left( \begin{smallmatrix}1 & \mu \\ 0 & 1 
\end{smallmatrix} \right)e_i\right\| \mid i=1,2,3 \right\}
\big\}=H_1\cup H_2 \cup H_3
\]
with $H_i=\left\{ \mu \in \R \mid  \| \big( \begin{smallmatrix}1 & \mu \\ 0 & 1 
\end{smallmatrix} \big) e_0 \|^2 = \| \big( \begin{smallmatrix}1 & \mu \\ 0 & 1 
\end{smallmatrix} \big) e_i\|^2 \right\}$. We shall show that $|H_i|<\infty$.

Let $(x_i,y_i)^T$ be the vector joining the endpoints of edge $e_i$,
where $0 \le i \le 3$. Without loss of generality, $y_i \ge 0$ and $x_i > 0$ 
if $y_i=0$. Then, for $1 \le i \le 3$,
\begin{align*} 
H_i & = \left\{ \mu \in \R \mid (x_0+\mu y_0)^2 + y_0^2 = 
(x_i + \mu y_i)^2 + y_i^2 \right\} \\
&  = \left\{ \mu \in \R \mid \mu^2(y_0^2- y_i^2) + 
2 \mu(x_0y_0 - x_iy_i)+x_0^2 + y_0^2 -x_i^2 -y_i^2=0 \right\}.
\end{align*}
Either $y_0 \ne y_i$, then $y_0^2-y_i^2 \ne 0$ and the last equation has at most two 
solutions in $\mu$. Or $y_0=y_i$ and $x_0 \ne x_i$. In the latter case we may have $y_0=y_i\ne 0$, whence $x_0 y_0 - x_i y_i \ne 0$ and the equation in the 
set above has a unique solution. Otherwise $y_0=y_i=0$ and $x_0 \ne x_i$. Since $x_0,x_i > 0$, the equation gives the contradiction $x_0^2-x_i^2=0$, and $H_i$ is empty. Altogether we obtain $|H_i| \le 2$.

(b) Now let $e_0$ be an edge of $T$ such that the corresponding vector $(x_0,y_0)^T$
satisfies $y_0 \ne 0$. Denote the edges of $T'$ by $e_1,e_2,e_3$ 
as above. Then
\begin{align*}
\left\{ \mu \in \R \mid \left( \begin{smallmatrix}1 & \mu \\ 0 & 1 
\end{smallmatrix} \right) T \cong T' \right\}
&\subseteq  \left\{ \mu \in \R \mid \| \left( \begin{smallmatrix}1 & \mu \\ 0 & 1 
\end{smallmatrix} \right) e_0 \|^2 \in \left\{\|e_1\|^2, \|e_2\|^2, \|e_3\|^2 \right\} \right\} \\
&= \left\{ \mu \in \R \mid (x_0+\mu y_0)^2+y_0^2   
\in \left\{\|e_1\|^2, \|e_2\|^2, \|e_3\|^2 \right\} \right\} \\
&= \left\{ \mu \in \R \mid  \mu^2y_0^2 + 2 \mu x_0 y_0 +x_0^2+y_0^2
\in \left\{\|e_1\|^2, \|e_2\|^2, \|e_3\|^2 \right\} \right\}.
\end{align*}
Since $y_0 \ne 0$, the last term is quadratic in $\mu$ again. Thus the cardinality of the last set is at most six.

(c) Pick edge vectors $(x_i,y_i)^T$, $i=1,2$, of $T$ such that $|y_1| \ne |y_2|$. Then
\begin{align*}
\left\{ \mu \in \R \mid 
\left( \begin{smallmatrix}1 & \mu \\ 0 & 1 \end{smallmatrix} \right) T 
\text{ is equilateral.} \right\} & 
\subseteq
\left\{ \mu \in \R \mid
\left\|\left( \begin{smallmatrix}1 & \mu \\ 0 & 1 \end{smallmatrix} \right)\left(\begin{smallmatrix} x_1 \\ y_1 \end{smallmatrix}\right)\right\|^2=\left\|\left( \begin{smallmatrix}1 & \mu \\ 0 & 1 \end{smallmatrix} \right)\left(\begin{smallmatrix} x_2 \\ y_2 \end{smallmatrix}\right)\right\|^2 \right\} \\
&\hspace{-10ex}=
\left\{ \mu \in \R \mid \mu^2\left(y_1^2-y_2^2\right)+2 \mu (x_1y_1-x_2y_2)+x_1^2+y_1^2-x_2^2-y_2^2=0\right\}.
\end{align*}
Since $y_1^2\ne y_2^2$, the last set consists of at most two numbers.
\end{proof}

Now we can provide sheared images of the tiling $\overline{\T}^\varepsilon$ that shall be stacked in order to form a tiling of the plane.

\begin{cor}\label{cor:Tn_eps}
For every $\varepsilon > 0$, there exist sheared images $\overline{\T}_n^\varepsilon= \left( \begin{smallmatrix}1 & \mu_n \\ 0 & 1 
\end{smallmatrix} \right) \overline{\T}^\varepsilon$, $n=1,2,\ldots$, of the tiling $\overline{\T}^\varepsilon$ such that \begin{enumerate}
\item[(a)] $\sum_{n=1}^\infty 2\sqrt{3}|\mu_n|< \varepsilon$,
\item[(b)] for all $n \ge 1$, $\overline{\T}_n^\varepsilon$ does not contain two distinct congruent triangles,
\item[(c)] for all $1 \le n' < n$, there are no congruent triangles $T \in \overline{\T}_n^\varepsilon$ and $T' \in \overline{\T}_{n'}^\varepsilon$,
\item[(d)] for all $n \ge 1$, $\overline{\T}_n^\varepsilon$ does not contain equilateral triangles. 
\end{enumerate} 
\end{cor}

\begin{proof}
We construct the tilings $\overline{\T}_n^\varepsilon$ by induction over $n$. Claim (a) will be obtained by choosing $\mu_n \in \left(-\frac{2^{-n}}{2\sqrt{3}}\varepsilon,\frac{2^{-n}}{2\sqrt{3}}\varepsilon\right)$.

Base case (construction of $\overline{\T}_1^\varepsilon$): By Corollary~\ref{cor:T_eps} and Lemma~\ref{lem:kong-endlich} (a) and (c), the set
\[
A= \bigcup_{T,T' \in \overline{\T}^\varepsilon,T \ne T'} \left\{\mu \in \mathbb{R} \mid
\left(\begin{smallmatrix}1 & \mu \\ 0 & 1 
\end{smallmatrix} \right)T\cong\left(\begin{smallmatrix}1 & \mu \\ 0 & 1 
\end{smallmatrix} \right)T' \text{ or } \left(\begin{smallmatrix}1 & \mu \\ 0 & 1 
\end{smallmatrix} \right)T \text{ is equilateral.}\right\}
\]
is at most countable, because $\overline{\T}^\varepsilon$ is countable. We pick $\mu_1 \in \left(-\frac{2^{-1}}{2\sqrt{3}}\varepsilon,\frac{2^{-1}}{2\sqrt{3}}\varepsilon\right) \setminus A$ and obtain $\overline{\T}_1^\varepsilon=\left(\begin{smallmatrix}1 & \mu_1 \\ 0 & 1 
\end{smallmatrix} \right)\overline{\T}^\varepsilon$. The choice of $\mu_1$ implies claims (b) and (d) for $n=1$.

Step of induction (construction of $\overline{\T}_n^\varepsilon$, $n \ge 2$): By Lemma~\ref{lem:kong-endlich} (b), the set
\[
B=\bigcup_{T \in \overline{\T}^\varepsilon,T' \in \overline{\T}_{1}^\varepsilon \cup \ldots \cup \overline{\T}_{n-1}^\varepsilon} \left\{\mu \in \mathbb{R} \mid
\left(\begin{smallmatrix}1 & \mu \\ 0 & 1 
\end{smallmatrix} \right)T \cong T' \right\}
\]
is at most countable, since $\overline{\T}^\varepsilon$ and $\overline{\T}_{1}^\varepsilon \cup \ldots \cup \overline{\T}_{n-1}^\varepsilon$ are countable. We fix some $\mu_n \in \left(-\frac{2^{-n}}{2\sqrt{3}}\varepsilon,\frac{2^{-n}}{2\sqrt{3}}\varepsilon\right) \setminus (A \cup B)$ this way defining $\overline{\T}_n^\varepsilon=\left(\begin{smallmatrix}1 & \mu_n \\ 0 & 1 
\end{smallmatrix} \right)\overline{\T}^\varepsilon$. We obtain the respective items of conditions (b) and (d) by $\mu_n \notin A$, and of (c) by $\mu_n \notin B$.
\end{proof}

\begin{proof}[Proof of Theorem~\ref{thm:triangles}]
A periodic vertex-to-vertex tiling $\T_{\text{periodic}}$ of the plane by equilateral triangles of edge length $2$ can be obtained by stacking copies of the tiling from Figure~\ref{fig:equilat_undistorted} under suitable vertical translations (or reflections through horizontal axes) above and below that tiling. Our modified tiling will be constructed similarly by stacking images of the tilings from Corollary~\ref{cor:Tn_eps}.

All tilings $\overline{\T}_n^\varepsilon=\left(\begin{smallmatrix}1 & \mu_n \\ 0 & 1 
\end{smallmatrix} \right)\overline{\T}^\varepsilon$ of $\overline{S}$ have the same mutual distances $a_{i+1}-a_i$ between adjacent vertices at the upper boundary $y=\sqrt{3}$ as $\overline{\T}^\varepsilon$. The same applies to the lower boundary $y=-\sqrt{3}$. The strip $\overline{S}$ is tiled by $\overline{\T}_1^\varepsilon$. Therefore the parallel strip $\overline{S}+\left(0,2\sqrt{3}\right)^T$ can be tiled by a suitable image $\varphi_2(\overline{\T}_2^\varepsilon)$ of $\overline{\T}_2^\varepsilon$ under a reflection through the horizontal coordinate axis and the translation $\left((\mu_1-\mu_2)\sqrt{3},2\sqrt{3}\right)^T$, so that $\overline{\T}_1^\varepsilon \cup \varphi_2(\overline{\T}_2^\varepsilon)$ is vertex-to-vertex. 
Similarly, we tile $\overline{S}+\left(0,-2\sqrt{3}\right)^T$ by a reflected and translated image $\varphi_{3}(\overline{\T}_{3}^\varepsilon)$ of $\overline{\T}_{3}^\varepsilon$, we tile $\overline{S}+\left(0,4\sqrt{3}\right)^T$ and $\overline{S}+\left(0,-4\sqrt{3}\right)^T$ by suitable translates $\varphi_4(\overline{\T}_4^\varepsilon)$ and $\varphi_{5}(\overline{\T}_{5}^\varepsilon)$ of $\overline{\T}_4^\varepsilon$ and $\overline{\T}_{5}^\varepsilon$ etc. This way we obtain the desired vertex-to-vertex tiling $\overline{\T}=\overline{\T}_1^\varepsilon \cup \varphi_2(\overline{\T}_2^\varepsilon) \cup \varphi_{3}(\overline{\T}_{3}^\varepsilon) \cup \ldots$ of $\mathbb{R}^2$, see Figure~\ref{fig:final}. The tiling $\overline{\T}$ consists of tiles of the same area $\sqrt{3}$, as $\overline{\T}_\varepsilon$ does by Corollary~\ref{cor:T_eps}. By Corollary~\ref{cor:Tn_eps} (b),(c),(d), the triangles of that tiling are mutually incongruent and not equilateral.

Finally, we shall see that $\overline{\T}$ is $2\varepsilon$-close to $\T_\text{periodic}$. By Corollary~\ref{cor:T_eps}, the absolute deviations of the second coordinates of vertices of $\overline{\T}^\varepsilon$ from those of $\T_\text{periodic}$ are smaller than $\varepsilon$. Neither the shear maps $\left(\begin{smallmatrix} 1 & \mu_n \\ 0 & 1 \end{smallmatrix}\right)$ nor the reflection through the horizontal coordinate axis do affect these deviations. Also the translations used for stacking the $\overline{\T}_n^\varepsilon$ keep the deviations of second coordinates unchanged. So the absolute deviations between second coordinates of $\overline{\T}$ and $\T_\text{periodic}$ are less than $\varepsilon$. Differences of the first coordinates of $\varphi_n(\overline{\T}_n^\varepsilon)=\varphi_n\left(\left(\begin{smallmatrix} 1 & \mu_n \\ 0 & 1 \end{smallmatrix}\right)\overline{\T}^\varepsilon\right)\subseteq \overline{\T}$ from the respective ones of $\T_\text{periodic}$ result from the differences between $\overline{\T}^\varepsilon$ and $\T_\text{periodic}$, which are smaller that $\varepsilon$ by Corollary~\ref{cor:T_eps}, from the horizontal shifts induced by the shear map $\left(\begin{smallmatrix} 1 & \mu_n \\ 0 & 1 \end{smallmatrix}\right)$, which are in $[-\sqrt{3}\mu_n,\sqrt{3}\mu_n]$, and the first coordinate of the translation vector of $\varphi_n$. If $\varphi_n(\overline{\T}_n^\varepsilon)$ and $\varphi_{n'}(\overline{\T}_{n'}^\varepsilon)$ tile consecutive strips $\overline{S}+\left(0,2k\sqrt{3}\right)^T$ and $\overline{S}+\left(0,2(k+1)\sqrt{3}\right)^T$, the respective translation vector between $\varphi_n(\overline{\T}_n^\varepsilon)$ and $\varphi_{n'}(\overline{\T}_{n'}^\varepsilon)$ is $\left(\pm(\mu_{n}-\mu_{n'})\sqrt{3},2\sqrt{3}\right)^T$, where the sign depends on whether $\varphi_{n'}$ or $\varphi_n$ involves a reflection. Adding the vectors up to get the translation vector from $\overline{\T}_1^\varepsilon$ to $\overline{\T}_n^\varepsilon$, we see that all absolute total deviations of first coordinates are smaller than
\[
\varepsilon+\sum_{n=1}^\infty 2\sqrt{3}|\mu_n| \stackrel{\text{(Corollary~\ref{cor:Tn_eps}(a))}}{<} 2\varepsilon.
\]
Consequently, $\overline{\T}$ is $2\varepsilon$-close to $\T_\text{periodic}$.

The proof of Theorem~\ref{thm:triangles} is complete.
\end{proof}


\section{Particular fair partitions of nearly equilateral triangles\label{sec:quadrangles}}

Figure~\ref{fig1} displays a partition of a triangle of edge lengths $a,b,c$ into three quadrangles. 
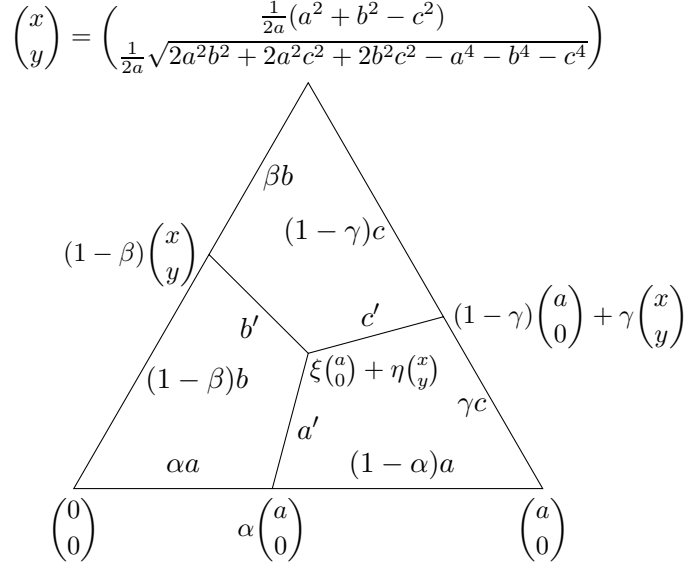
\begin{figure}
\begin{center}
\begin{tikzpicture}[xscale=1.039,yscale=.9]

\draw 
  (0,0)--(6,0)--(3,6)--cycle
  (2.54,0)--(3,2)--(1.73,3.46)
  (3,2)--(4.73,2.54)
  (0,0) node[below] {$\genfrac(){0pt}{0}{0}{0}$}
  (2.54,0) node[below] {$\alpha\genfrac(){0pt}{0}{a}{0}$}
  (6,0) node[below] {$\genfrac(){0pt}{0}{a}{0}$}
  (1.73,3.46) node[left] {$(1-\beta)\genfrac(){0pt}{0}{x}{y}$}
  (2.9,2.12) node[below right] {$\xi\genfrac(){0pt}{1}{a}{0}+\eta\genfrac(){0pt}{1}{x}{y}$}
  (4.73,2.54) node[right] {$(1-\gamma)\genfrac(){0pt}{0}{a}{0}+\gamma\genfrac(){0pt}{0}{x}{y}$}
  (3,6) node[above] { $\genfrac(){0pt}{0}{x}{y}=\genfrac(){0pt}{0}{\frac{1}{2a}(a^2+b^2-c^2)}{\frac{1}{2a}\sqrt{2a^2b^2+2a^2c^2+2b^2c^2-a^4-b^4-c^4}}$}
  (1.4,.1) node[above] {\large $\alpha a$}
  (4.2,0) node[above] {\large $(1-\alpha)a$}
  (.8,1.6) node[right] {\large $(1-\beta)b$}
  (2.3,4.6) node[right] {\large $\beta b$}
  (5.4,1.2) node[left] {\large $\gamma c$}
  (4.1,3.8) node[left] {\large $(1-\gamma)c$}
  (2.74,.9) node[right] {\large $a'$}
  (2.5,2.7) node[below left] {\large $b'$}
  (3.8,2.27) node[above] {\large $c'$}
  ;
\end{tikzpicture}
\end{center}
\caption{A dissection of a triangle into quadrangles.\label{fig1}}
\end{figure}
The independent parameters $a,b,c > 0$ (fulfilling the triangle inequalities), $\alpha,\beta,\gamma \in (0,1)$ and $\xi,\eta > 0$ with $\xi+\eta < 1$ describe the tiling up to congruence under Euclidean isometries. In the illustration the actual values are $a=b=c=1$, $\alpha=\beta=\gamma=1-\frac{\sqrt{3}}{3}$ and $\xi=\eta=\frac{1}{3}$. Then the quadrangles are congruent, the sizes of their inner angles are $\frac{\pi}{3}$, $\frac{7\pi}{12}$, $\frac{2\pi}{3}$ and $\frac{5\pi}{12}$ and their perimeter is $p_0=1+\sqrt{2}-\frac{\sqrt{6}}{3}$. 

Now we consider a perturbed situation.

\begin{lem}\label{lem:subdivision}
There exists $\delta > 0$ and a procedure of dissecting arbitrary triangles with edge lengths $a,b,c \in (1-\delta,1+\delta)$ into three convex quadrangles as in Figure~\ref{fig1} with the following properties: 
\begin{itemize}
\item[(i)]
The three quadrangles have the same perimeter
\[
p_0=1+\sqrt{2}-\frac{\sqrt{6}}{3}
\]
and the same area. 
\item[(ii)]
If the original triangle is not equilateral then the quadrangles are mutually incongruent.
\item[(iii)]
The shape of any of the quadrangles determines the shape of the original triangle (up to isometry).
\end{itemize}
\end{lem}

\begin{proof}
For arbitrary values of $\alpha,\beta,\gamma$ sufficiently close to $1-\frac{\sqrt{3}}{3}$, the choice
\[
\xi=\frac{1-2\alpha-\gamma+3\alpha\gamma}{3(1-\alpha-\beta-\gamma+\alpha\beta+\alpha\gamma+\beta\gamma)},\quad
\eta=\frac{1-\beta-2\gamma+3\beta\gamma}{3(1-\alpha-\beta-\gamma+\alpha\beta+\alpha\gamma+\beta\gamma)}
\]
ensures that the three quadrangles from Figure~\ref{fig1} have the same area. (Indeed, this condition is invariant under affine transformations. Hence this is enough to be checked for the equilateral triangle with $a=b=c=1$.) It remains to show that there are $\alpha=\alpha(a,b,c)$, $\beta=\beta(a,b,c)$ and $\gamma=\gamma(a,b,c)$ close to $1-\frac{\sqrt{3}}{3}$ such that the perimeters of the quadrangles are $p_0$. This is equivalent to the three equations
\begin{align*}
\left(\alpha a+ a'+b'+(1-\beta)b\right)-p_0 &=0,\\
\left(\beta b+b'+c'+(1-\gamma)c\right)-p_0 &=0,\\
\left(\gamma c+ c'+a'+(1-\alpha)a\right)-p_0 &=0.
\end{align*}
Writing $a',b',c'$ in terms of $a,b,c,\alpha,\beta,\gamma,\xi,\eta$ and then $\xi,\eta$ in terms of $\alpha,\beta,\gamma$ as above, the equations amount to a non-linear system
\begin{align}
f_1(a,b,c,\alpha,\beta,\gamma)&=0,\nonumber\\
f_2(a,b,c,\alpha,\beta,\gamma)&=0,\label{eq:system}\\
f_3(a,b,c,\alpha,\beta,\gamma)&=0,\nonumber
\end{align} 
that has to be solved for $(\alpha,\beta,\gamma)$ in dependence of $(a,b,c)$. Since \eqref{eq:system} is satisfied for $a=b=c=1$, $\alpha=\beta=\gamma=1-\frac{\sqrt{3}}{3}$ (see above) and since
\[
\det \left.\left( 
\begin{array}{ccc}
\frac{\partial}{\partial \alpha}f_1 &
\frac{\partial}{\partial \beta}f_1 &
\frac{\partial}{\partial \gamma}f_1 \\
\frac{\partial}{\partial \alpha}f_2 &
\frac{\partial}{\partial \beta}f_2 &
\frac{\partial}{\partial \gamma}f_2 \\
\frac{\partial}{\partial \alpha}f_3 &
\frac{\partial}{\partial \beta}f_3 &
\frac{\partial}{\partial \gamma}f_3
\end{array}
\right)\right|_{a=b=c=1,\alpha=\beta=\gamma=1-\frac{\sqrt{3}}{3}}=2\sqrt{2}+\sqrt{3}-2\sqrt{6}\ne 0
\]
(as can be shown by a computer algebra system such as Maple 2019), the implicit function theorem shows that the system \eqref{eq:system} has a unique and continuous solution $(\alpha(a,b,c),\beta(a,b,c),\gamma(a,b,c))$ in some neighbourhood of $(a,b,c)=(1,1,1)$ such that $\alpha(1,1,1)=\beta(1,1,1)=\gamma(1,1,1)=1-\frac{\sqrt{3}}{3}$. This proves part (i).

For (ii), we use the notations from the left-hand part of Figure~\ref{fig2}.
\begin{figure}
\begin{center}
\begin{tikzpicture}[xscale=.53,yscale=.5]

\draw
  (-8,2.98) node
  {
  \begin{tikzpicture}[xscale=.53,yscale=.5]
  \draw 
  (0,0)--(6,0)--(3,6)--cycle
  (2.54,0)--(3,2)--(1.73,3.46)
  (3,2)--(4.73,2.54)
  (0,0) node[below] {$A$}
  (2.54,0) node[below] {$C'$}
  (6,0) node[below] {$B$}
  (1.73,3.46) node[left] {$B'$}
  (3.25,2.6) node {$M$}
  (4.7,2.7) node[right] {$A'$}
  (3,6) node[above] { $C$}
  (1.8,1.35) node {\large $Q_1$}
  (4,1.1) node {\large $Q_2$}
  (3,4) node {\large $Q_3$}
  ;
\end{tikzpicture}
  };
  
\fill[lightgray]
  (0,0)--(2.52,0)--(3,2)--(1.73,3.46)--cycle
  ;
\draw 
  (0,0)--(6,0)--(3,6)--cycle
  (2.54,0)--(3,2)--(1.73,3.46)
  (3,2)--(4.73,2.54)
  (0,0) node[below] {$\genfrac(){0pt}{0}{0}{0}$}
  (2.54,0) node[below] {$\genfrac(){0pt}{0}{\hat{a}}{0}$}
  (6,0) node[below] {$\varrho\genfrac(){0pt}{0}{\hat{a}}{0}$}
  (1.7,3.46) node[left] {$\genfrac(){0pt}{0}{\hat{x}}{\hat{y}}$}
  (2.7,2.2) node[below right] {$\genfrac(){0pt}{0}{\hat{z}}{\hat{w}}$}
  (4.73,2.54) node[right] {$(1-\tau)\varrho\genfrac(){0pt}{0}{\hat{a}}{0}+\tau\sigma\genfrac(){0pt}{0}{\hat{x}}{\hat{y}}$}
  (3,6) node[above] { $\sigma\genfrac(){0pt}{0}{\hat{x}}{\hat{y}}$}
  ;
\end{tikzpicture}
\end{center}
\caption{Notations for the proofs of (ii) (on the left) and (iii) (on the right).\label{fig2}}
\end{figure}
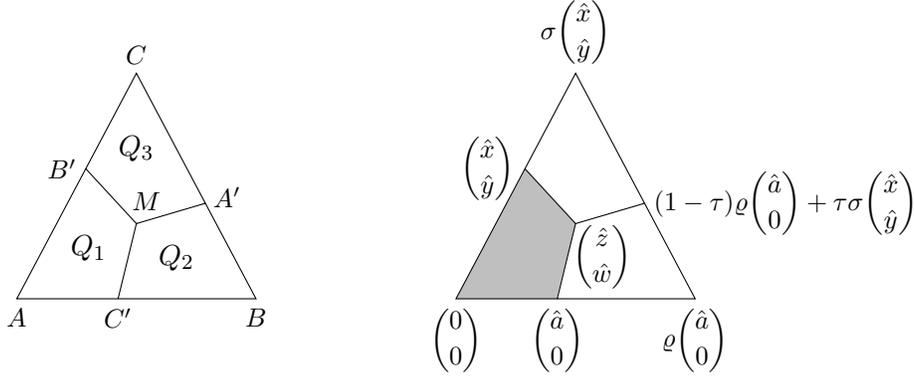
We suppose that the quadrangles $Q_1$ and $Q_2$ are congruent and have to show that the underlying triangle is equilateral. Since our dissection is close to the undistorted one (if $\delta$ is chosen sufficiently small), we know that the angles of $Q_1$ at $A$, $C'$, $M$ and $B'$ must correspond to the angles of $Q_2$ at $B$, $A'$, $M$ and $C'$, respectively. We obtain
$MB' \cong MC' \cong MA'$
and, using pairs of supplementary angles at $A'$, $B'$ and $C'$,
$|\angle(AC'M)|=|\angle(BA'M)|=|\angle(CB'M)|=\pi-|\angle(AB'M)|=\pi-|\angle(BC'M)|=\pi-|\angle(CA'M)|$.
So the quadrangles $Q_3$ and $Q_1$ are related by 
\[
A'M \cong B'M,\quad B'M \cong C'M,\quad 
\angle(CA'M) \cong \angle(AB'M),\quad \angle(CB'M) \cong \angle(AC'M).
\]
Taking into account that $Q_1$ and $Q_3$ both have the same perimeter $p_0$, we see that $Q_3 \cong Q_1$ as well. So all inner angles of the underlying triangle are congruent. This proves (ii).

For (iii), note that the quadrangles correspond to a small perturbation of Figure~\ref{fig1} with $a=b=c=1$ and $\alpha=\beta=\gamma=1-\frac{\sqrt{3}}{3}$. Such a quadrangle is represented up to congruence by the vertices $(0,0)^T$, $(\hat{a},0)^T\approx(\hat{a}_0,0)^T$, $(\hat{z},\hat{w})^T\approx(\hat{z}_0,\hat{w}_0)^T$ and $(\hat{x},\hat{y})^T\approx(\hat{x}_0,\hat{y}_0)^T$ (see the grey quadrangle in the right-hand part of Figure~\ref{fig2}),
where
\[
\hat{a}_0=1-\frac{\sqrt{3}}{3},\quad
\hat{x}_0=\frac{\sqrt{3}}{6},\quad
\hat{y}_0=\frac{1}{2},\quad
\hat{z}_0=\frac{1}{2},\quad
\hat{w}_0=\frac{\sqrt{3}}{6}
\]
stand for the undistorted situation.
Since we consider triangles close to equilateral triangles, we know 
that the smallest inner angle of our quadrangle is one of the dissected triangle, we can assume that the triangle and the other two quadrangles are positioned as in Figure~\ref{fig2}. Now it is enough to show that the parameters $\varrho$ and $\sigma$ are uniquely determined by the given coordinates $\hat{a},\hat{x},\hat{y},\hat{z},\hat{w}$ (provided that they are sufficiently close to $\hat{a}_0,\hat{x}_0,\hat{y}_0,\hat{z}_0,\hat{w}_0$), because this shows that the shape of our quadrangle determines the dissected triangle up to congruence. 

Note that the parameters $\varrho,\sigma,\tau$ in the undistorted situation are
\[
\varrho_0=\frac{1}{1-\frac{\sqrt{3}}{3}}=\frac{3+\sqrt{3}}{2}, \quad \sigma_0=\frac{1}{1-\left(1-\frac{\sqrt{3}}{3}\right)}=\sqrt{3}, \quad
\tau_0=1-\frac{\sqrt{3}}{3}.
\]
The area of the grey quadrangle is
\[
\text{area}_1=\frac{1}{2}\det\left(\genfrac(){0pt}{0}{\hat{a}}{0}-\genfrac(){0pt}{0}{0}{0},
\genfrac(){0pt}{0}{\hat{x}}{\hat{y}}-\genfrac(){0pt}{0}{0}{0}\right)
+\frac{1}{2}\det\left(\genfrac(){0pt}{0}{\hat{x}}{\hat{y}}-\genfrac(){0pt}{0}{\hat{z}}{\hat{w}},
\genfrac(){0pt}{0}{\hat{a}}{0}-\genfrac(){0pt}{0}{\hat{z}}{\hat{w}}\right),
\]
and the areas of the quadrangles at the vertices $\varrho\genfrac(){0pt}{1}{\hat{a}}{0}$ and 
$\sigma\genfrac(){0pt}{1}{\hat{x}}{\hat{y}}$ are
\begin{align*}
\text{area}_2=\frac{1}{2}\det&\left(\left((1-\tau)\varrho\genfrac(){0pt}{0}{\hat{a}}{0}+\tau\sigma\genfrac(){0pt}{0}{\hat{x}}{\hat{y}}\right)-\varrho\genfrac(){0pt}{0}{\hat{a}}{0},
\genfrac(){0pt}{0}{\hat{a}}{0}-\varrho\genfrac(){0pt}{0}{\hat{a}}{0}\right)\\
&+\frac{1}{2}\det\left(\genfrac(){0pt}{0}{\hat{a}}{0}-\genfrac(){0pt}{0}{\hat{z}}{\hat{w}},
\left((1-\tau)\varrho\genfrac(){0pt}{0}{\hat{a}}{0}+\tau\sigma\genfrac(){0pt}{0}{\hat{x}}{\hat{y}}\right)-\genfrac(){0pt}{0}{\hat{z}}{\hat{w}}\right),\\
\text{area}_3=\frac{1}{2}\det&\left(\genfrac(){0pt}{0}{\hat{x}}{\hat{y}}-\sigma\genfrac(){0pt}{0}{\hat{x}}{\hat{y}},\left((1-\tau)\varrho\genfrac(){0pt}{0}{\hat{a}}{0}+\tau\sigma\genfrac(){0pt}{0}{\hat{x}}{\hat{y}}\right)-\sigma\genfrac(){0pt}{0}{\hat{x}}{\hat{y}}\right)\\
&+\frac{1}{2}\det\left(\left((1-\tau)\varrho\genfrac(){0pt}{0}{\hat{a}}{0}+\tau\sigma\genfrac(){0pt}{0}{\hat{x}}{\hat{y}}\right)-\genfrac(){0pt}{0}{\hat{z}}{\hat{w}},
\genfrac(){0pt}{0}{\hat{x}}{\hat{y}}-\genfrac(){0pt}{0}{\hat{z}}{\hat{w}}\right).
\end{align*}
Since the quadrangles have equal area, the functions
\begin{align*}
&\hat{f}_1(\hat{a},\hat{x},\hat{y},\hat{z},\hat{w},\varrho,\sigma,\tau)=\text{area}_1-\text{area}_2,\\
&\hat{f}_2(\hat{a},\hat{x},\hat{y},\hat{z},\hat{w},\varrho,\sigma,\tau)=\text{area}_1-\text{area}_3
\end{align*} 
satisfy
\[
\hat{f}_1(\hat{a},\hat{x},\hat{y},\hat{z},\hat{w},\varrho,\sigma,\tau)=\hat{f}_2(\hat{a},\hat{x},\hat{y},\hat{z},\hat{w},\varrho,\sigma,\tau)=0.
\]
Moreover, the construction of the quadrangles shows in particular that the perimeter of the second one is $p_0$. That is, the function
\begin{align*}
\hat{f}_3(\hat{a},\hat{x},\hat{y},\hat{z},\hat{w},\varrho,\sigma,\tau)=
\left\|\left((1-\tau)\varrho\genfrac(){0pt}{0}{\hat{a}}{0}+\tau\sigma\genfrac(){0pt}{0}{\hat{x}}{\hat{y}}\right)-\varrho\genfrac(){0pt}{0}{\hat{a}}{0}\right\|+\left\|\genfrac(){0pt}{0}{\hat{a}}{0}-\varrho\genfrac(){0pt}{0}{\hat{a}}{0}\right\|&\\
+\left\| \genfrac(){0pt}{0}{\hat{a}}{0}-\genfrac(){0pt}{0}{\hat{z}}{\hat{w}}\right\|+\left\|\left((1-\tau)\varrho\genfrac(){0pt}{0}{\hat{a}}{0}+\tau\sigma\genfrac(){0pt}{0}{\hat{x}}{\hat{y}}\right)-\genfrac(){0pt}{0}{\hat{z}}{\hat{w}}\right\|-p_0&
\end{align*}
satisfies
\[
\hat{f}_3(\hat{a},\hat{x},\hat{y},\hat{z},\hat{w},\varrho,\sigma,\tau)=0
\]
as well.

Now we apply the implicit function theorem: since $(\hat{a},\hat{x},\hat{y},\hat{z},\hat{w},\rho,\sigma,\tau)$ is sufficiently close to $(\hat{a}_0,\hat{x}_0,\hat{y}_0,\hat{z}_0,\hat{w}_0,\rho_0,\sigma_0,\tau_0)$ (by our construction of the dissection, provided $\delta$ is chosen sufficiently small), since
\begin{align*}
&\hat{f}_1(\hat{a}_0,\hat{x}_0,\hat{y}_0,\hat{z}_0,\hat{w}_0,\varrho_0,\sigma_0,\tau_0)=0,\\
&\hat{f}_2(\hat{a}_0,\hat{x}_0,\hat{y}_0,\hat{z}_0,\hat{w}_0,\varrho_0,\sigma_0,\tau_0)=0,\\
&\hat{f}_3(\hat{a}_0,\hat{x}_0,\hat{y}_0,\hat{z}_0,\hat{w}_0,\varrho_0,\sigma_0,\tau_0)=0
\end{align*}
(as can easily be seen) and since
\[
\det \left.\left( 
\begin{array}{ccc}
\frac{\partial}{\partial \varrho}\hat{f}_1 &
\frac{\partial}{\partial \sigma}\hat{f}_1 &
\frac{\partial}{\partial \tau}\hat{f}_1 \\
\frac{\partial}{\partial \varrho}\hat{f}_2 &
\frac{\partial}{\partial \sigma}\hat{f}_2 &
\frac{\partial}{\partial \tau}\hat{f}_2 \\
\frac{\partial}{\partial \varrho}\hat{f}_3 &
\frac{\partial}{\partial \sigma}\hat{f}_3 &
\frac{\partial}{\partial \tau}\hat{f}_3
\end{array}
\right)\right|_{\genfrac{}{}{0pt}{1}{(\hat{a},\hat{x},\hat{y},\hat{z},\hat{w},\varrho,\sigma,\tau)=}{(\hat{a}_0,\hat{x}_0,\hat{y}_0,\hat{z}_0,\hat{w}_0,\varrho_0,\sigma_0,\tau_0)}}=\frac{\sqrt{6}}{48}-\frac{\sqrt{2}}{24}\ne 0
\]
(as can be shown by a computer algebra system such as Maple 2019), the
system
\begin{align*}
\hat{f}_1(\hat{a},\hat{x},\hat{y},\hat{z},\hat{w},\varrho,\sigma,\tau)=0&,\\
\hat{f}_2(\hat{a},\hat{x},\hat{y},\hat{z},\hat{w},\varrho,\sigma,\tau)=0&,\\
\hat{f}_3(\hat{a},\hat{x},\hat{y},\hat{z},\hat{w},\varrho,\sigma,\tau)=0&
\end{align*}
has a unique solution for $(\varrho,\sigma,\tau)$ depending on $(\hat{a},\hat{x},\hat{y},\hat{z},\hat{w})$. In particular, $\varrho$ and $\sigma$ are uniquely determined by $\hat{a}$, $\hat{x}$, $\hat{y}$, $\hat{z}$ and $\hat{w}$. This completes the proof of (iii).
\end{proof}

\begin{proof}[Proof of Theorem~\ref{thm:quadrangles}]
Let $\delta$ be as in Lemma~\ref{lem:subdivision}. By Theorem~\ref{thm:triangles}, there is a dissection of the plane into mutually incongruent and non-equilateral triangles of equal area with edge lengths in $(1-\delta,1+\delta)$. We apply Lemma~\ref{lem:subdivision} to all these triangles. By Lemma~\ref{lem:subdivision}(i) this yields a dissection of the plane into convex quadrangles of equal area and equal perimeter. By Lemma~\ref{lem:subdivision}(ii), two quadrangles are incongruent if they originate from the same triangle, since the triangles are non-equilateral. By Lemma~\ref{lem:subdivision}(iii), two quadrangles are incongruent if they originate from different triangles, since the triangles are incongruent.
\end{proof}

\begin{figure}
\begin{center}
\begin{tikzpicture}[xscale=.346,yscale=.3]

\begin{scope}
  \clip (1.5,0) rectangle (28.5,24)
  ;
  \begin{scope}[cm={1,0,0,1,(0,0)}]
    \draw 
      (0,0)--(6,0)--(3,6)--cycle
      (2.54,0)--(3,2)--(1.73,3.46)
      (3,2)--(4.73,2.54)
      (4.27,3.46)--(6,4)--(7.27,2.54)
      (6,4)--(6.46,6)
      ;
  \end{scope}
  \begin{scope}[cm={1,0,0,1,(6,0)}]
    \draw 
      (0,0)--(6,0)--(3,6)--cycle
      (2.54,0)--(3,2)--(1.73,3.46)
      (3,2)--(4.73,2.54)
      (4.27,3.46)--(6,4)--(7.27,2.54)
      (6,4)--(6.46,6)
      ;
  \end{scope}
  \begin{scope}[cm={1,0,0,1,(12,0)}]
    \draw 
      (0,0)--(6,0)--(3,6)--cycle
      (2.54,0)--(3,2)--(1.73,3.46)
      (3,2)--(4.73,2.54)
      (4.27,3.46)--(6,4)--(7.27,2.54)
      (6,4)--(6.46,6)
      ;
  \end{scope}
  \begin{scope}[cm={1,0,0,1,(18,0)}]
    \draw 
      (0,0)--(6,0)--(3,6)--cycle
      (2.54,0)--(3,2)--(1.73,3.46)
      (3,2)--(4.73,2.54)
      (4.27,3.46)--(6,4)--(7.27,2.54)
      (6,4)--(6.46,6)
      ;
  \end{scope}
  \begin{scope}[cm={1,0,0,1,(24,0)}]
    \draw 
      (0,0)--(6,0)--(3,6)--cycle
      (2.54,0)--(3,2)--(1.73,3.46)
      (3,2)--(4.73,2.54)
      (4.27,3.46)--(6,4)--(7.27,2.54)
      (6,4)--(6.46,6)
      ;
  \end{scope}
  \begin{scope}[cm={1,0,0,1,(-3,6)}]
    \draw 
      (0,0)--(6,0)--(3,6)--cycle
      (2.54,0)--(3,2)--(1.73,3.46)
      (3,2)--(4.73,2.54)
      (4.27,3.46)--(6,4)--(7.27,2.54)
      (6,4)--(6.46,6)
      ;
  \end{scope}
  \begin{scope}[cm={1,0,0,1,(3,6)}]
    \draw 
      (0,0)--(6,0)--(3,6)--cycle
      (2.54,0)--(3,2)--(1.73,3.46)
      (3,2)--(4.73,2.54)
      (4.27,3.46)--(6,4)--(7.27,2.54)
      (6,4)--(6.46,6)
      ;
  \end{scope}
  \begin{scope}[cm={1,0,0,1,(9,6)}]
    \draw 
      (0,0)--(6,0)--(3,6)--cycle
      (2.54,0)--(3,2)--(1.73,3.46)
      (3,2)--(4.73,2.54)
      (4.27,3.46)--(6,4)--(7.27,2.54)
      (6,4)--(6.46,6)
      ;
  \end{scope}
  \begin{scope}[cm={1,0,0,1,(15,6)}]
    \draw 
      (0,0)--(6,0)--(3,6)--cycle
      (2.54,0)--(3,2)--(1.73,3.46)
      (3,2)--(4.73,2.54)
      (4.27,3.46)--(6,4)--(7.27,2.54)
      (6,4)--(6.46,6)
      ;
  \end{scope}
  \begin{scope}[cm={1,0,0,1,(21,6)}]
    \draw 
      (0,0)--(6,0)--(3,6)--cycle
      (2.54,0)--(3,2)--(1.73,3.46)
      (3,2)--(4.73,2.54)
      (4.27,3.46)--(6,4)--(7.27,2.54)
      (6,4)--(6.46,6)
      ;
  \end{scope}
  \begin{scope}[cm={1,0,0,1,(27,6)}]
    \draw 
      (0,0)--(6,0)--(3,6)--cycle
      (2.54,0)--(3,2)--(1.73,3.46)
      (3,2)--(4.73,2.54)
      (4.27,3.46)--(6,4)--(7.27,2.54)
      (6,4)--(6.46,6)
      ;
  \end{scope}
  \begin{scope}[cm={1,0,0,1,(0,12)}]
    \draw 
      (0,0)--(6,0)--(3,6)--cycle
      (2.54,0)--(3,2)--(1.73,3.46)
      (3,2)--(4.73,2.54)
      (4.27,3.46)--(6,4)--(7.27,2.54)
      (6,4)--(6.46,6)
      ;
  \end{scope}
  \begin{scope}[cm={1,0,0,1,(6,12)}]
    \draw 
      (0,0)--(6,0)--(3,6)--cycle
      (2.54,0)--(3,2)--(1.73,3.46)
      (3,2)--(4.73,2.54)
      (4.27,3.46)--(6,4)--(7.27,2.54)
      (6,4)--(6.46,6)
      ;
  \end{scope}
  \begin{scope}[cm={1,0,0,1,(12,12)}]
    \draw 
      (0,0)--(6,0)--(3,6)--cycle
      (2.54,0)--(3,2)--(1.73,3.46)
      (3,2)--(4.73,2.54)
      (4.27,3.46)--(6,4)--(7.27,2.54)
      (6,4)--(6.46,6)
      ;
  \end{scope}
  \begin{scope}[cm={1,0,0,1,(18,12)}]
    \draw 
      (0,0)--(6,0)--(3,6)--cycle
      (2.54,0)--(3,2)--(1.73,3.46)
      (3,2)--(4.73,2.54)
      (4.27,3.46)--(6,4)--(7.27,2.54)
      (6,4)--(6.46,6)
      ;
  \end{scope}
  \begin{scope}[cm={1,0,0,1,(24,12)}]
    \draw 
      (0,0)--(6,0)--(3,6)--cycle
      (2.54,0)--(3,2)--(1.73,3.46)
      (3,2)--(4.73,2.54)
      (4.27,3.46)--(6,4)--(7.27,2.54)
      (6,4)--(6.46,6)
      ;
  \end{scope}
  \begin{scope}[cm={1,0,0,1,(-3,18)}]
    \draw 
      (0,0)--(6,0)--(3,6)--cycle
      (2.54,0)--(3,2)--(1.73,3.46)
      (3,2)--(4.73,2.54)
      (4.27,3.46)--(6,4)--(7.27,2.54)
      (6,4)--(6.46,6)
      ;
  \end{scope}
  \begin{scope}[cm={1,0,0,1,(3,18)}]
    \draw 
      (0,0)--(6,0)--(3,6)--cycle
      (2.54,0)--(3,2)--(1.73,3.46)
      (3,2)--(4.73,2.54)
      (4.27,3.46)--(6,4)--(7.27,2.54)
      (6,4)--(6.46,6)
      ;
  \end{scope}
  \begin{scope}[cm={1,0,0,1,(9,18)}]
    \draw 
      (0,0)--(6,0)--(3,6)--cycle
      (2.54,0)--(3,2)--(1.73,3.46)
      (3,2)--(4.73,2.54)
      (4.27,3.46)--(6,4)--(7.27,2.54)
      (6,4)--(6.46,6)
      ;
  \end{scope}
  \begin{scope}[cm={1,0,0,1,(15,18)}]
    \draw 
      (0,0)--(6,0)--(3,6)--cycle
      (2.54,0)--(3,2)--(1.73,3.46)
      (3,2)--(4.73,2.54)
      (4.27,3.46)--(6,4)--(7.27,2.54)
      (6,4)--(6.46,6)
      ;
  \end{scope}
  \begin{scope}[cm={1,0,0,1,(21,18)}]
    \draw 
      (0,0)--(6,0)--(3,6)--cycle
      (2.54,0)--(3,2)--(1.73,3.46)
      (3,2)--(4.73,2.54)
      (4.27,3.46)--(6,4)--(7.27,2.54)
      (6,4)--(6.46,6)
      ;
  \end{scope}
  \begin{scope}[cm={1,0,0,1,(27,18)}]
    \draw 
      (0,0)--(6,0)--(3,6)--cycle
      (2.54,0)--(3,2)--(1.73,3.46)
      (3,2)--(4.73,2.54)
      (4.27,3.46)--(6,4)--(7.27,2.54)
      (6,4)--(6.46,6)
      ;
  \end{scope}
  \draw
    (0,24)--(30,24)
    ;
\end{scope}  
\end{tikzpicture}
\end{center}

\caption{The resulting tiling is close to a periodic tiling by quadrangles.
\label{fig:final_quadrangles}} 
\end{figure}
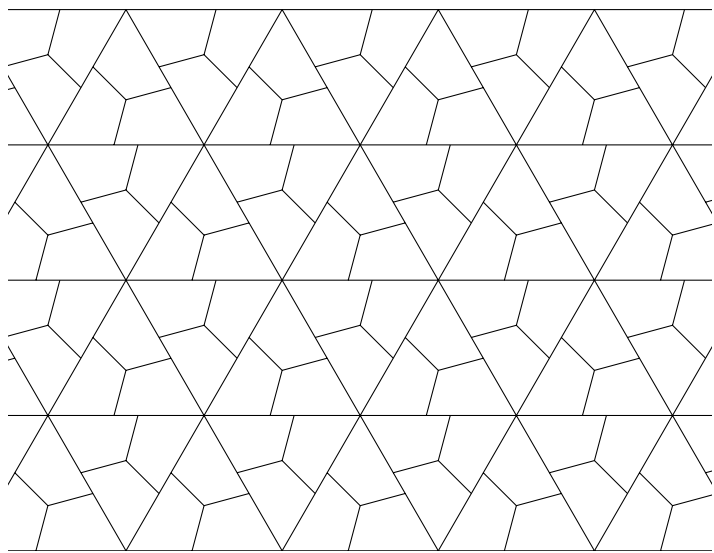

\begin{rem}
Our construction gives a tiling that is close to the periodic tiling illustrated in Figure~\ref{fig:final_quadrangles}. It is not vertex-to-vertex. It remains open if Theorem~\ref{thm:quadrangles} can be sharpened in this direction.
\end{rem}


\section*{Acknowledgments}
Both authors express their gratitude to R.\ Nandakumar for providing 
several interesting problems.



\end{document}